\documentclass[12pt,reqno]{amsart}
\usepackage{amsmath, amsthm, amscd, amsfonts, amssymb, graphicx, color}
\usepackage[bookmarksnumbered, colorlinks, plainpages]{hyperref}
\hypersetup{colorlinks=true,linkcolor=red, anchorcolor=green, citecolor=cyan, urlcolor=red, filecolor=magenta, pdftoolbar=true}

\newtheorem{theorem}{Theorem}[section]
\newtheorem{lemma}[theorem]{Lemma}
\newtheorem{proposition}[theorem]{Proposition}
\newtheorem{corollary}[theorem]{Corollary}
\theoremstyle{definition}
\newtheorem{definition}[theorem]{Definition}
\newtheorem{example}[theorem]{Example}

\theoremstyle{remark}
\newtheorem{remark}[theorem]{Remark}
\numberwithin{equation}{section}

\begin{document}
	\setcounter{page}{1}
	\title[Controlled continuous $\ast$-$g$-Frames in Hilbert $C^{\ast}$-Modules]{Controlled continuous $\ast$-$g$-Frames in Hilbert $C^{\ast}$-Modules}
	
	\author[ M. Ghiati, M. Rossafi, M. Mouniane, H. Labrigui, A. Touri]{M'hamed Ghiati$^{1*}$, Mohamed Rossafi$^{2}$,  Mohammed Mouniane$^{3}$, Hatim Labrigui$^{4}$ and Abdeslam Touri$^{5}$}
	
	\address{$^{1}$	Laboratory of Analysis, Geometry and Application - LAGA, Department of Mathematics, Faculty of Sciences, University of Ibn Tofail, P. O.  Box 133 Kenitra, Morocco}
	\email{mhamed.ghiati@uit.ac.ma}

	\address{$^{2}$LaSMA Laboratory, Department of Mathematics, Faculty of Sciences Dhar El Mahraz, University Sidi Mohamed Ben Abdellah, P. O.  Box 1796 Fez Atlas, Morocco}
	\email{rossafimohamed@gmail.com}
	
	\address{$^{3}$	Laboratory of Analysis, Geometry and Application - LAGA, Department of Mathematics, Faculty of Sciences, University of Ibn Tofail, P. O.  Box 133 Kenitra, Morocco}
	\email{mouniane.mohammed@uit.ac.ma}

	\address{$^{4}$	 Laboratory Partial Differential Equations, Spectral Algebra and Geometry, Department of Mathematics, Faculty of Sciences, University of Ibn Tofail, P. O.  Box 133 Kenitra, Morocco}
	\email{hlabrigui75@gmail.com}

	\address{$^{5}$	 Laboratory Partial Differential Equations, Spectral Algebra and Geometry, Department of Mathematics, Faculty of Sciences, University of Ibn Tofail, P. O.  Box 133 Kenitra, Morocco}
	\email{touri.abdo68@gmail.com}
%	
%	\address{$^{1}$	Laboratory of Analysis, Geometry and Application - LAGA, Department of Mathematics, Faculty of Sciences, University of Ibn Tofail, P. O.  Box 133 Kenitra, Morocco}
%	\email{mhamed.ghiati@uit.ac.ma; mouniane.mohammed@uit.ac.ma}
%	
%	\address{$^{2}$LaSMA Laboratory, Department of Mathematics, Faculty of Sciences Dhar El Mahraz, University Sidi Mohamed Ben Abdellah, P. O.  Box 1796 Fez Atlas, Morocco}
%	\email{rossafimohamed@gmail.com; mohamed.rossafi@usmba.ac.ma}
%	
%
%	
%	\address{$^{3}$	 Laboratory Partial Differential Equations, Spectral Algebra and Geometry, Department of Mathematics, Faculty of Sciences, University of Ibn Tofail, P. O.  Box 133 Kenitra, Morocco}
%		
%	\email{hlabrigui75@gmail.com; touri.abdo68@gmail.com}

	\subjclass[2010]{41A58,  42C15, 46L05.}
	
	\keywords{Continuous frames; continuous $\ast$-$g$-frames; Controlled continuous $\ast$-$g$-frames; $C^{\ast}$-algebra; Hilbert $C^{\ast}$-modules.\\$^*$Corresponding author: M'hamed Ghiati (mhamed.ghiati@uit.ac.ma).}

	\begin{abstract}
	The frame theory is dynamic and exciting with various pure and applied mathematics applications. In this paper, we introduce and study the concept of Controlled  Continuous $\ast$-$g$-Frames in Hilbert $C^{\ast}$-Modules, which is a generalization of discrete controlled $\ast$-$g$-Frames in Hilbert $C^{\ast}$-Modules. Also, we give some properties.
	\end{abstract} 
\maketitle
	
	\baselineskip=15pt
	%-----------------------------------%
	\section{Introduction and preliminaries}\label{section1}
	Duffin and Schaeffer introduced the concept of frames in Hilbert spaces  \cite{Duf} in 1952 to study some severe problems in the nonharmonic Fourier series. After the fundamental paper \cite{13} by Daubechies, Grossman and Meyer, frames theory began to be widely used, particularly in the more specialized context of wavelet frames and Gabor frames \cite{Gab}.\\
	Hilbert $C^{\ast}$-module arose as generalizations of the notion of Hilbert space. The basic idea was to consider modules over $C^{\ast}$-algebras instead of linear spaces and to allow the inner product to take values in the $C^{\ast}$-algebras (See  \cite{Lance1995,Pas}).\\
	Continuous frames are defined by Ali, Antoine, and Gazeau \cite{STAJP}. Gabardo and Han in \cite{14} called these kinds of frames, frames associated with measurable spaces.\\
	The theory of frames has been extended from Hilbert spaces to Hilbert  $C^{\ast}$-modules. For more details see   \cite{Ghiati2022, Frank,Ghiati2021, mjpaaa,MR2,ARAN,MR1,r20}.\\
	In the following, ${U}$ is  Hilbert $C^{*}$-module, $End^{*}_{\mathcal{A}}({U},{V})$ is the set of all adjointable  operators from ${U}$ into ${V}$ and $End^{*}_{\mathcal{A}}({U},{U})$ is abbreviated to $End^{*}_{\mathcal{A}}({U})$ ,  $\mathcal{GL}(U)$ is the set of all bounded linear operators which have bounded inverses and $\mathcal{GL^{+}}(U)$ is the set of all positive operators in $\mathcal{GL}(U)$. 
	The operators $\mathcal{C,C'}\in\mathcal{GL^{+}}(U)$,  and $ \Lambda:=\{  \Lambda_{w}\in{End^{*}_{\mathcal{A}}({U},{V}_{w})} , w\in \Omega \}$ is a sequence of bounded operators.\\
	We introduce the notion of Controlled Continuous $\ast$-$g$-Frame in Hilbert $C^{\ast}$-Modules, which is a generalization of discrete controlled $\ast$-$g$-Frames in Hilbert $C^{\ast}$-Modules given by Zahra Ahmadi Moosavi and Akbar Nazari \cite{zahra}, and we establish some new results.\\
	The paper is organized as follows; we continue this introductory section by  recalling briefly the definitions and basic properties of $C^{\ast}$-algebra, Hilbert $C^{\ast}$-modules. Our reference for $C^{\ast}$-algebras is \cite{Dav,Con}.\\
	In Section \ref{section2}, we introduce some properties of continuous ${\ast}$-$g$-frame. In Section \ref{section3}, we discuss the controlled continuous ${\ast}$-$g$-frame in Hilbert $C^{\ast}$-module. The Duality of continuous ${\ast}$-$g$-frame is considered in Section \ref{section4}. In Section \ref{section5}, the stability problem for continuous ${\ast}$-$g$-frame in Hilbert $C^{\ast}$-module is treated. The last section is consecrated for some properties of $(C,C^{'})$-controlled continuous $\ast$-$g$-frames.
%	{In this work, we briefly recall the definitions and basic properties of $C^{\ast}$-algebra, Hilbert $\mathcal{A}$-modules. Our reference for $C^{\ast}$-algebras is \cite{{Dav},{Con}}.}
	\begin{definition}\cite{Con}.	
Let $ \mathcal{A} $ be a unital $C^{\ast}$-algebra and ${U}$ be a left $ \mathcal{A} $-module, such that the linear structures of $\mathcal{A}$ and $ {U} $ are compatible. ${U}$ is a pre-Hilbert $\mathcal{A}$-module if ${U}$ is equipped with an $\mathcal{A}$-valued inner product $\langle.,.\rangle_{\mathcal{A}} :{U}\times{U}\rightarrow\mathcal{A}$, such that is sesquilinear, positive definite and respects the module action. In other words,
		\begin{itemize}
			\item [(i)] $ \langle x,x\rangle_{\mathcal{A}}\geq0 $ for all $ x\in{U} $ and $ \langle x,x\rangle_{\mathcal{A}}=0$ if and only if $x=0$.
			\item [(ii)] $\langle ax+y,z\rangle_{\mathcal{A}}=a\langle x,z\rangle_{\mathcal{A}}+\langle y,z\rangle_{\mathcal{A}}$ for all $a\in\mathcal{A}$ and $x,y,z\in{U}$.
			\item[(iii)] $ \langle x,y\rangle_{\mathcal{A}}=\langle y,x\rangle_{\mathcal{A}}^{\ast} $ for all $x,y\in{U}$.
		\end{itemize}	 
		For $x\in{U}, $ we define $||x||=||\langle x,x\rangle_{\mathcal{A}}||^{\frac{1}{2}}$. If ${U}$ is complete with $||.||$, it is called a Hilbert $\mathcal{A}$-module or a Hilbert $C^{\ast}$-module over $\mathcal{A}$. For every $a$ in $C^{\ast}$-algebra $\mathcal{A}$, we have $|a|=(a^{\ast}a)^{\frac{1}{2}}$ and the $\mathcal{A}$-valued norm on ${U}$ is defined by $|x|=\langle x, x\rangle_{\mathcal{A}}^{\frac{1}{2}}$ for all $x\in{U}$.\\

		Let ${U}$ and ${V}$ be two Hilbert $\mathcal{A}$-modules, A map $T:{U}\rightarrow{V}$ is said to be adjointable if there exists a map $T^{\ast}:{V}\rightarrow{U}$ such that $\langle Tx,y\rangle_{\mathcal{A}}=\langle x,T^{\ast}y\rangle_{\mathcal{A}}$ for all $x\in{U}$ and $y\in{V}$.
		
		We reserve the notation $End_{\mathcal{A}}^{\ast}({U},{V})$ for the set of all adjointable operators from ${U}$ to ${V}$ and $End_{\mathcal{A}}^{\ast}({U},{U})$ is abbreviated to $End_{\mathcal{A}}^{\ast}({U})$.	
	\end{definition}
	The following lemmas will be used to prove our main results.
	\begin{lemma} \label{l1} \cite{Pas}.
		Let  ${U}$ be Hilbert $\mathcal{A}$-module. If $T\in End_{\mathcal{A}}^{\ast}({U})$, then $$\langle Tx,Tx\rangle\leq\|T\|^{2}\langle x,x\rangle \qquad \forall x\in{U}.$$
	\end{lemma}
	
	\begin{lemma} \label{l2} \cite{Ara}.
		Let ${U}$ and ${V}$ two Hilbert $\mathcal{A}$-modules and $T\in End^{\ast}({U},{V})$. Then the following statements are equivalent
		\begin{itemize}
			\item [(i)] $T$ is surjective.
			\item [(ii)] $T^{\ast}$ is bounded below with respect to norm, i.e., there is $m>0$ such that $ m\|x\| \leq \|T^{\ast}x\|$ for all $x\in{V}$.
			\item [(iii)] $T^{\ast}$ is bounded below with respect to the inner product, i.e., there is $m'>0$ such that $m'\langle x,x\rangle \leq \langle T^{\ast}x,T^{\ast}x\rangle $ for all $x\in{V}$.
		\end{itemize}
	\end{lemma}
	\begin{lemma} \label{3} \cite{Deh}.
		Let ${U}$ and ${V}$ two Hilbert $\mathcal{A}$-modules and $T\in End^{\ast}({U},{V})$. Then
		\begin{itemize}
			\item [(i)] If $T$ is injective and $T$ has closed range, then the adjointable map $T^{\ast}T$ is invertible and $$\|(T^{\ast}T)^{-1}\|^{-1}\leq T^{\ast}T\leq\|T\|^{2}.$$
			\item  [(ii)]	If $T$ is surjective, then the adjointable map $TT^{\ast}$ is invertible and $$\|(TT^{\ast})^{-1}\|^{-1}\leq TT^{\ast}\leq\|T\|^{2}.$$
		\end{itemize}	
	\end{lemma}
	
		\begin{lemma}\cite{Lance1995}
		For self-adjoint $f\in C(X)$, the following are equivalent
		\begin{itemize}
			\item [(1)] $f\geq 0$
			\item  [(2)]	For all $t\geq \|f\|$, we have $\|f-t\|\leq t$
			\item  [(3)]	For all least one $t\geq \|f\|$, we have $\|f-t\|\leq t$
		\end{itemize}	
	\end{lemma}

	\section{Some Properties of Continuous $\ast$-$g$-Frames in Hilbert $C^{\ast}$-Modules}\label{section2}
	Let $X$ be a Banach space, $(\Omega,\mu)$ a measure space, and function $f:\Omega\to X$ a measurable function. Integral of the Banach-valued function $f$ has defined Bochner and others. Most properties of this integral are similar to those of the integral of real-valued functions. Because every $C^{\ast}$-algebra and Hilbert $C^{\ast}$-module is a Banach space thus, we can use this integral and its properties.
	
	Let $(\Omega,\mu)$ be a measure space, let $U$ and $V$ be two Hilbert $C^{\ast}$-modules, $\{V_{w}\}_{w\in\Omega}$  is a sequence of subspaces of V, and $End_{\mathcal{A}}^{\ast}(U,V_{w})$ is the collection of all adjointable $\mathcal{A}$-linear maps from $U$ into $V_{w}$.
	We define
	\begin{equation*}
		\oplus_{w\in\Omega}V_{w}=\left\{x=\{x_{w}\}_{w\in\Omega}: x_{w}\in V_{w}, \left\|\int_{\Omega}|x_{w}|^{2}d\mu(w)\right\|<\infty\right\}.
	\end{equation*}
	For any $x=\{x_{w}\}_{w\in\Omega}$ and $y=\{y_{w}\}_{w\in\Omega}$, if the $\mathcal{A}$-valued inner product is defined by $\langle x,y\rangle=\int_{\Omega}\langle x_{w},y_{w}\rangle d\mu(w)$, the norm is defined by $\|x\|=\|\langle x,x\rangle\|^{\frac{1}{2}}$, the $\oplus_{w\in\Omega}V_{w}$ is a Hilbert $C^{\ast}$-module.\cite{Lance1995}.
	Let $GL^{+}(U)$ be the set for all positive bounded linear invertible operators on $U$ with the bounded inverse.

	\begin{definition}\cite{RR}
		We call $\{\Lambda_{w}\in End_{\mathcal{A}}^{\ast}(U,V_{w}): w\in\Omega\}$ a continuous $\ast$-$g$-frame for Hilbert $C^{\ast}$-module $U$ with respect to $\{V_{w}: w\in\Omega\}$ if
		\begin{itemize}
			\item for any $x\in U$, the function $\tilde{x}:\Omega\rightarrow V_{w}$ defined by $\tilde{x}(w)=\Lambda_{w}x$ is measurable;
			\item there exist two strictly nonzero elements $A$ and $B$ in $\mathcal{A}$ such that
			\begin{equation} \label{eqd2}
				A\langle x,x\rangle A^{\ast} \leq\int_{\Omega}\langle\Lambda_{w}x,\Lambda_{w}x\rangle d\mu(w)\leq B\langle x,x\rangle B^{\ast}, \forall x\in U.
			\end{equation}
		\end{itemize}
		The elements $A$ and $B$ are called continuous $\ast$-$g$-frame bounds. \\
		If $A=B$ we call this continuous $\ast$-$g$-frame a continuous tight g-frame, and if $A=B=1_{\mathcal{A}}$ it is called a continuous Parseval $\ast$-$g$-frame. If only the right-hand inequality of \eqref{eqd2} is satisfied, we call $\{\Lambda_{w}: w\in\Omega\}$ a 
		continuous $\ast$-$g$-Bessel sequence for $U$ with respect to $\{V_{w}: w\in\Omega\}$ with Bessel bound $B$.\\
		The continuous $\ast$-$g$-frame operator $S$ on $U$ is 
		\begin{equation*}
			Sx=\int_{\Omega}\Lambda^{\ast}_{\omega}\Lambda_{\omega}xd\mu (\omega)
		\end{equation*}
	\end{definition}
	
	\begin{theorem}\label{t1}
		Let $\{\Lambda_{w}\}_{w\in\Omega}\in End_{\mathcal{A}}^{\ast}(U,V_{w})$, such that $\|\int_{\Omega}\langle\Lambda_{w}x,\Lambda_{w}x\rangle d\mu(w)\|< \infty $ , then $\{\Lambda_{w}\}_{w\in\Omega}$ be a continuous $\ast$-$g$-frame for $U$ with respect to $\{V_{w}: w\in\Omega\}$ if and only if there exist a constants $A$ and $B$ such that for any $x\in U$ :
		\begin{equation}\label{eqt1}
			\|A^{-1}\|^{-2}\|\langle x,x\rangle \| \leq \bigg\| \int_{\Omega}\langle\Lambda_{w}x,\Lambda_{w}x\rangle d\mu(w)\bigg\| \leq \|B^{2}\|\|\langle x,x\rangle \|
		\end{equation}
	\end{theorem}
	\begin{proof}
		By the definition of  continuous $\ast$-$g$-frame, we have
		$$\langle x,x\rangle \leq A^{-1}\langle Sx,x\rangle (A^{\ast})^{-1} \text{ and }  
		\langle Sx,x\rangle \leq B\langle x,x\rangle B^{\ast}.$$
		Hence 
		\begin{equation*}
			\|A^{-1}\|^{-2}\|\langle x,x\rangle \| \leq\bigg\| \int_{\Omega}\langle\Lambda_{w}x,\Lambda_{w}x\rangle d\mu(w)\bigg\| \leq \|B^{2}\|\|\langle x,x\rangle \|.
		\end{equation*}
		For the converse, assume that \eqref{eqt1} holds, for any $x\in U$, we define : \\
		$Sx:= \int_{\Omega}\Lambda^{\ast}_{w}\Lambda_{w}xd\mu (w)$, then :
		\begin{align*}
			\|Sx\|^{4}=&\|\langle Sx, Sx \rangle \|^{2}\\
			&=\|\langle Sx, \int_{\Omega}\Lambda^{\ast}_{w}\Lambda_{w}xd\mu(w)\rangle \bigg\|^{2}\\
			&=\bigg\|\int_{\Omega}\langle \Lambda_{w}Sx, \Lambda_{w}x\rangle d\mu(w)\bigg\|^{2}\\
			&\leq \bigg\|\int_{\Omega}\langle \Lambda_{w}Sx, \Lambda_{w}Sx\rangle d\mu(w)\bigg\|\bigg\|\int_{\Omega}\langle \Lambda_{w}x, \Lambda_{w}x\rangle d\mu(w)\bigg\|\\
			&\leq \|B\|^{2}\|Sx\|^{2}\|B\|^{2}\|x\|^{2}
		\end{align*}
		Hence  
		$$\|Sx\|^{2}\leq \|B\|^{4}\|x\|^{2}.$$
		It is easy to check that $\langle Sx,y\rangle = \langle x,Sy\rangle$, so $S$ is bounded and $S=S^{\ast}$, from $\langle Sx,x\rangle =\int_{\Omega}\langle \Lambda_{w}x, \Lambda_{w}x\rangle d\mu(w)\geq 0$ it follows that $ 0\leq S$,\\
		Now $\langle S^{\frac{1}{2}}x,S^{\frac{1}{2}}x \rangle\leq \|S^{\frac{1}{2}}\|^{2}\langle x,x\rangle $ \\
		On the other hand, we have
		\begin{equation*}
			\|(S^{\frac{1}{2}})^{\ast}S^{\frac{1}{2}}\|\langle x,x\rangle = \|S\|\langle x,x\rangle,
		\end{equation*}
		therefore, we get 
		\begin{equation*}
			\langle Sx,x\rangle = \langle S^{\frac{1}{2}}x,S^{\frac{1}{2}}x \rangle \leq \|S\|\langle x,x\rangle \leq \|B\|^{2}1_{\mathcal{A}}\langle x,x\rangle = (\|B\|1_{\mathcal{A}})\langle x,x\rangle (\|B\|1_{\mathcal{A}})^{\ast},
		\end{equation*}
		and by \eqref{eqt1} we have $\|A^{-1}\|^{-2}\|\langle x,x\rangle \| \leq \| S^{\frac{1}{2}} x\|^{2}$.\\
		We conclude that $\|A^{-1}\|^{-1}\|x\| \leq \| S^{\frac{1}{2}} x\|$
		so by Lemma \ref{l2} we obtain lower bound for $\Lambda$, this shows that $\Lambda$ is a continuous $\ast$-$g$-frame for $U$
		with respect to $\{V_{w}: w\in\Omega\}$.
	\end{proof}
	\begin{proposition}\label{p2}
		Let $\Lambda= \{\Lambda_{w}\in End_{\mathcal{A}}^{\ast}(U,V_{w}): w\in\Omega\}$ and $\Theta = \{\theta_{w}\in End_{\mathcal{A}}^{\ast}(U,V_{w}): w\in\Omega\}$ be a two continuous $\ast$-$g$-Bessel sequences for $U$ with respect to $\{V_{w}: w\in\Omega\}$ with bounds $B_{\Lambda} , B_{\Theta}$ and $\Gamma=\{\Gamma_{\omega}\}_{\omega \in \Omega} \in l^{\infty}(\mathbb{C})$, then the operator  $L=L_{\Gamma , \Lambda , \Theta} : U \longrightarrow U$ such that $L_{\Gamma , \Lambda , \Theta} x = \int_{\Omega}\Gamma_{\omega}\Lambda^{\ast}_{\omega}\theta_{w}xd\mu(w)$ is well defined bounded operator.
	\end{proposition}
	\begin{proof}
		From the definition of $\Lambda$, $\Theta$ and $\Gamma$, we have for any $x,y \in U$
		\begin{align*}
			\bigg\|\int_{\Omega}\Gamma_{\omega}\Lambda^{\ast}_{\omega}\theta_{w}xd\mu(w)\bigg\|^{2}&=\underset{y \in U, \|y\|\leq1}{\sup}\bigg\|\langle \int_{\Omega}\Gamma_{\omega}\Lambda^{\ast}_{\omega}\theta_{w}xd\mu(w),y\rangle\bigg\|^{2}\\
			&=\underset{y \in U, \|y\|\leq1}{\sup}\bigg\|\int_{\Omega}\langle \Gamma_{\omega}\theta_{w}xd\mu(w),\Lambda_{\omega}y\rangle\bigg\|^{2}\\
			&\leq \underset{y \in U, \|y\|\leq1}{\sup}\bigg\|\int_{\Omega}\langle \Gamma_{\omega}\theta_{w}x,\Gamma_{\omega}\theta_{w}x\rangle d\mu(w)\bigg\|\bigg\|\int_{\Omega}\langle \Lambda_{\omega}y,\Lambda_{\omega}y\rangle d\mu(w)\bigg\|.
		\end{align*} 
		On other hand, we have 
		\begin{align*}
			\int_{\Omega}\langle \Gamma_{\omega}\theta_{w}x,\Gamma_{\omega}\theta_{w}x\rangle d\mu(w) &=\int_{\Omega}|\Gamma_{\omega}|^{2}\langle \theta_{w}x,\theta_{w}x\rangle d\mu(w)\\
			&\leq \|\Gamma_{\omega}\|^{2}_{\infty}\int_{\Omega}\langle \theta_{w}x,\theta_{w}x\rangle d\mu(w)\\
			&\leq \|\Gamma_{\omega}\|^{2}_{\infty}B_{\Theta}\langle x,x\rangle B^{\ast}_{\Theta} .
		\end{align*}
		Hence 
		\begin{align*}
			\bigg\|\int_{\Omega}\Gamma_{\omega}\Lambda^{\ast}_{\omega}\theta_{w}xd\mu(w)\bigg\|^{2} &\leq \underset{y \in U, \|y\|\leq1}{\sup}\|\Gamma_{\omega}\|^{2}_{\infty}\|B_{\Theta}\|^{2}\|\langle x,x\rangle\|\|B_{\Lambda}\|^{2}\|\langle y,y\rangle\|\\
			&=\|\Gamma_{\omega}\|^{2}_{\infty}\|B_{\Theta}\|^{2}\|\langle x,x\rangle\|\|B_{\Lambda}\|^{2},
		\end{align*}
		then $L_{\Gamma , \Lambda , \Theta}$ is well defined and 
		\begin{equation*}
			\|L_{\Gamma , \Lambda , \Theta}\|\leq \|\Gamma_{\omega}\|^{2}_{\infty}\|B_{\Theta}\|^{2}\|B_{\Lambda}\|^{2}
		\end{equation*}
	\end{proof}
	The map $L$ in the above proposition is called a continuous $\ast$-$g$-multiplier of $\Lambda$,$\Theta$ and $\Gamma$.
	\begin{lemma}
		Let $\Lambda= \{\Lambda_{w}\in End_{\mathcal{A}}^{\ast}(U,V_{w}): w\in\Omega\}$ and $\Theta = \{\theta_{w}\in End_{\mathcal{A}}^{\ast}(U,V_{w}): w\in\Omega\}$ be a continuous $\ast$-$g$-sequence for $U$ with respect to $\{V_{w}: w\in\Omega\}$ with bounds $B_{\Lambda} , B_{\Theta}$ and $\Gamma=\{\Gamma_{\omega}\}_{\omega \in \Omega} \in l^{\infty}(\mathbb{C})$, then the operator :\\ $L=L_{\Gamma , \Lambda , \Theta} : U \longrightarrow U$ such that $\langle Lx,y\rangle = \int_{\Omega}\Gamma_{\omega}\langle \Theta_{\omega}x,\Lambda y\rangle d\mu(w)$ is well defined and $(L_{\Gamma , \Lambda , \Theta})^{\ast}=L_{\bar{\Gamma} , \Lambda , \Theta}$.
	\end{lemma}
	\begin{proof}
		By proposion \ref{p2}, $L$ is well defined.\\
		We have
		\begin{align*}
			\langle x, (L_{\Gamma , \Lambda , \Theta})^{\ast}y\rangle &=\langle (
			L_{\Gamma , \Lambda , \Theta} x, y\rangle\\
			&= \int_{\Omega}\Gamma_{\omega}\langle \Theta_{\omega}x,\Lambda_{\omega} y\rangle d\mu(w)\\
			&= \int_{\Omega}\langle \Theta_{\omega}x, \bar{\Gamma_{\omega}}\Lambda_{\omega}y\rangle d\mu(w)\\
			&= \int_{\Omega}\langle x,\Theta^{\ast}_{\omega} \bar{\Gamma_{\omega}}\Lambda_{\omega}y\rangle d\mu(w)\\
			&= \int_{\Omega}\langle x,\bar{\Gamma_{\omega}}\Theta^{\ast}_{\omega} \Lambda_{\omega}y\rangle d\mu(w)\\
			&= \langle x,\int_{\Omega}\bar{\Gamma_{\omega}}\Theta^{\ast}_{\omega} \Lambda_{\omega}yd\mu(w)\rangle\\
			&=\langle x,L_{\bar{\Gamma} , \Lambda , \Theta}y\rangle.
		\end{align*}
	\end{proof}
	\section{Controlled continuous $\ast$-$g$-frames}\label{section3}
	In this section, we will introduce the concepts of controlled continuous $\ast$-$g$-frames in Hilbert $C^{\ast}$-modules.
	\begin{definition}\label{d3}
		Let $C,C^{'} \in GL^{+}(U)$, the family $\Lambda= \{\Lambda_{w}\in End_{\mathcal{A}}^{\ast}(U,V_{w}): w\in\Omega\}$ be called a $(C$-$C^{'})$-controlled continuous $\ast$-$g$-frame for Hilbert $C^{\ast}$-module $U$ with respect to  $\{V_{w}: w\in\Omega\}$ if there exist two strictly nonzero elements A,B in $\mathcal{A}$ such that :
		\begin{equation}\label{eqd3}
			A\langle x,x\rangle A^{\ast}\leq\int_{\Omega}\langle\Lambda_{w}Cx,\Lambda_{w}C^{'}x\rangle d\mu(w)\leq B\langle x,x\rangle B^{\ast}, \forall x\in U.
		\end{equation}
		$A$ and $B$ are called the $(C$-$C^{'})$-controlled continuous $\ast$-$g$-frames bounds.\\
		If $C^{'}=I$ then we call $\Lambda$ a $C$-controlled continuous $\ast$-$g$-frames for $U$ with respect to  $\{V_{w}: w\in\Omega\}$.\\
	\end{definition}
	\begin{example}
		Let $U=\{(a_{n})_{n\in \mathbb{N}^{\ast}} \subset \mathbb{C}\quad / \quad \sum_{n\in \mathbb{N}^{\ast}}|a_{n}|^{2}<+\infty  \}$ and let $\quad$  $\mathcal{A}=\{(a_{n})_{n\in \mathbb{N}^{\ast}} \subset \mathbb{C}\quad/ \quad(a_{n})_{n\in \mathbb{N}^{\ast}} \text{ is  bouned  }   \}$.
		It's clear that $\mathcal{A}$ is a unitary $C^{\ast}$-algebra.\\
		We define the inner product

		\begin{align*}
			U\times U \qquad \quad & \longrightarrow  \mathcal{A} \\ 
			((a_{n})_{n\in \mathbb{N}^{\ast}},(b_{n})_{n\in \mathbb{N}^{\ast}}) & \longrightarrow (a_{n}\bar{b_{n}})_{n\in \mathbb{N}^{\ast}}
		\end{align*}
		This inner product makes $U$ a $C^{\ast}$-module on  $\mathcal{A}$.
		We define $C,C^{'} \in End^{\ast}_{\mathcal{A}}(U)$ by,
		\begin{align*}
			C:U&\longrightarrow U\\
			(a_{n})_{n\in \mathbb{N}^{\ast}}&\longrightarrow (\alpha a_{n})_{n\in \mathbb{N}^{\ast}}
		\end{align*}
		\begin{align*}
			C':U&\longrightarrow U\\
			(a_{n})_{n\in \mathbb{N}^{\ast}}&\longrightarrow (\beta a_{n})_{n\in \mathbb{N}^{\ast}}
		\end{align*}
		where $\alpha$ and $\beta$ are in $\mathbb{R}^{\ast +}$.\\
		Now, we consider a measure space ($\Omega =\left[ 0,1\right] ,d\mu $), whose $d\mu $ is a Lebesgue measure  restraint on the interval $\left[ 0,1\right] $.\\
		Let $\{\Lambda _{w}\}_{w\in \Omega }$ be a sequence of operators defined by
		\begin{align*}
			\Lambda_{\omega}:\quad U&\longrightarrow U\\
			(a_{n})_{n\in \mathbb{N}^{\ast}}&\longrightarrow (
			\frac{\omega a_{n}}{n})_{n\in \mathbb{N}^{\ast}}
		\end{align*}
		These operators are continuous because they are bounded.\\
		We have,
		\begin{align*}
			\int_{\Omega}\langle \alpha(
			\frac{\omega a_{n}}{n})_{n\in \mathbb{N}^{\ast}},\beta(
			\frac{\omega a_{n}}{n})_{n\in \mathbb{N}^{\ast}}\rangle_{\mathcal{A}} d\mu(\omega)&=\alpha\beta\int_{\Omega}\omega^{2}d\mu(\omega)\langle (
			\frac{a_{n}}{n})_{n\in \mathbb{N}^{\ast}},(\frac{ a_{n}}{n})_{n\in \mathbb{N}^{\ast}}\rangle_{\mathcal{A}}\\
			&=\frac{\alpha\beta}{3}(\frac{1}{n^{2}})_{n\in \mathbb{N}^{\ast}}.\langle (a_{n})_{n\in \mathbb{N}^{\ast}},(a_{n})_{n\in \mathbb{N}^{\ast}}\rangle_{\mathcal{A}}\\
			&=\sqrt{\frac{\alpha\beta}{3}}(\frac{1}{n})_{n\in \mathbb{N}^{\ast}}\langle (a_{n})_{n\in \mathbb{N}^{\ast}},(a_{n})_{n\in \mathbb{N}^{\ast}}\rangle_{\mathcal{A}}\sqrt{\frac{{\alpha\beta}}{3}}(\frac{1}{n})_{n\in \mathbb{N}^{\ast}}
		\end{align*}
		So, we have
		\begin{equation*}
			\int_{\Omega}\langle \alpha(
			\frac{\omega a_{n}}{n})_{n\in \mathbb{N}^{\ast}},\beta(
			\frac{\omega a_{n}}{n})_{n\in \mathbb{N}^{\ast}}\rangle_{\mathcal{A}} d\mu(\omega)\leq \sqrt{\alpha\beta}(\frac{1}{n})_{n\in \mathbb{N}^{\ast}}\langle (a_{n})_{n\in \mathbb{N}^{\ast}},(a_{n})_{n\in \mathbb{N}^{\ast}}\rangle_{\mathcal{A}}\sqrt{\alpha\beta}(\frac{1}{n})_{n\in \mathbb{N}^{\ast}}
		\end{equation*}
		\begin{equation*}
			\frac{\sqrt{\alpha\beta}}{4}(\frac{1}{n})_{n\in \mathbb{N}^{\ast}}\langle (a_{n})_{n\in \mathbb{N}^{\ast}},(a_{n})_{n\in \mathbb{N}^{\ast}}\rangle_{\mathcal{A}}\frac{\sqrt{\alpha\beta}}{4}(\frac{1}{n})_{n\in \mathbb{N}^{\ast}}\leq 
			\int_{\Omega}\langle \alpha(
			\frac{\omega a_{n}}{n})_{n\in \mathbb{N}^{\ast}},\beta(
			\frac{\omega a_{n}}{n})_{n\in \mathbb{N}^{\ast}}\rangle_{\mathcal{A}} d\mu(\omega).
		\end{equation*}
		Which shows that $\{\Lambda_{\omega}\}_{\omega \in \Omega}$ is a $(C$-$C^{'})$-controlled continuous $\ast$-$g$-frames for $U$ with respect to $\{ U_{\omega}, \omega \in \Omega\}$ where $U_{\omega}=U \quad for \, all \; \omega \in \Omega$.
	\end{example}
\begin{theorem} \label{2.3}
	Let $\{\Lambda_{w}\in End_{\mathcal{A}}^{\ast}(U,V_{w}): w\in\Omega\}$ be a $(C$-$C^{'})$-controlled continuous $\ast$-$g$-frame for $U$, with lower and upper bounds $A$ and $B$, respectively. Then the $(C$-$C^{'})$-controlled continuous $\ast$-$g$-frame transform $T:U\rightarrow\oplus_{w\in\Omega}V_{w}$ defined by $T(C^{'}\Lambda_{w}^{\ast}\Lambda_{w}C)^{\frac{1}{2}}x=\{C^{'}\Lambda_{w}^{\ast}\Lambda_{w}Cx: w\in\Omega\}$ is injective and  adjointable, and  has a closed range  with $\|T\|\leq\|B\|$. The adjoint operator $T^{\ast}$ is surjective, given by $T^{\ast}(C^{'}\Lambda_{w}^{\ast}\Lambda_{w}C)^{\frac{1}{2}}x=\int_{\Omega}(C^{'}\Lambda_{w}^{\ast}\Lambda_{w}C)x_{w}d\mu(w)$, where $x=\{x_{w}\}_{w\in\Omega}$.
\end{theorem}

\begin{proof}
	Let $x\in U$. By the definition of a $(C$-$C^{'})$-controlled continuous $\ast$-$g$-frame for $U$, we have
	\begin{equation*}
		A\langle x,x\rangle A^{\ast}\leq\int_{\Omega}\langle\Lambda_{w}Cx,\Lambda_{w}C^{'}x\rangle d\mu(w)\leq B\langle x,x\rangle B^{\ast},
	\end{equation*}
hence
\begin{equation*}
	A\langle x,x\rangle A^{\ast}\leq\langle \int_{\Omega}C^{'}\Lambda_{w}^{\ast}\Lambda_{w}Cxd\mu(w),x\rangle \leq B\langle x,x\rangle B^{\ast}.
\end{equation*}
Therefore
\begin{equation*}
	A\langle x,x\rangle A^{\ast}\leq\langle T^{*}Tx,x\rangle \leq B\langle x,x\rangle B^{\ast}.
\end{equation*}
	So
	\begin{equation}\label{1}
		A\langle x,x\rangle A^{\ast}\leq\langle Tx,Tx\rangle\leq B\langle x,x\rangle B^{\ast}.
	\end{equation}
	If $Tx=0$ then $\langle x,x\rangle=0$ and so $x=0$, i.e., $T$ is injective.
	
	We now show that the range of $T$ is closed. Let $\{Tx_{n}\}_{n\in\mathbb{N}}$ be a sequence in the range of $T$ such that $\lim_{n\rightarrow\infty}Tx_{n}=y.$\\
	By \eqref{1} we have, for $n, m\in\mathbb{N}$,
	\begin{equation*}
		\|A\langle x_{n}-x_{m},x_{n}-x_{m}\rangle A^{\ast}\|\leq\|\langle T(x_{n}-x_{m}),T(x_{n}-x_{m})\rangle\|=\|T(x_{n}-x_{m})\|^{2}.
	\end{equation*}
	Since $\{Tx_{n}\}_{n\in\mathbb{N}}$ is Cauchy sequence in $U$,
	$\|A\langle x_{n}-x_{m},x_{n}-x_{m}\rangle A^{\ast}\|\rightarrow0$, as $n,m\rightarrow\infty.$\\	
	Note that for $n, m\in\mathbb{N}$,
	\begin{eqnarray}
		\|\langle x_{n}-x_{m},x_{n}-x_{m}\rangle\|= \|A^{-1}A\langle x_{n}-x_{m},x_{n}-x_{m}\rangle A^{\ast}(A^{\ast})^{-1}\|\\ \leq  \|A^{-1}\|^{2}\|A\langle x_{n}-x_{m},x_{n}-x_{m}\rangle A^{\ast}\|.
	\end{eqnarray}
	Therefore the sequence $\{x_{n}\}_{n\in\mathbb{N}}$ is Cauchy and hence there exists $x\in U$ such that $x_{n}\rightarrow x$ as $n\rightarrow\infty$. Again by \eqref{1}, we have $\|T(x_{n}-x)\|^{2}\leq\|B\|^{2}\|\langle x_{n}-x,x_{n}-x\rangle\|$.
	
	Thus $\|Tx_{n}-Tx\|\rightarrow0$ as $n\rightarrow\infty$ implies that $Tx=y$. It concludes that the range of $T$ is closed.
	
	For all $x\in U$, $y=\{y_{w}\}\in\oplus_{w\in\Omega}V_{w}$, we have
	\begin{align*}
		\langle T(C^{'}\Lambda_{w}^{\ast}\Lambda_{w}C)^{\frac{1}{2}}x,y\rangle
		&=\int_{\Omega}\langle (C^{'}\Lambda_{w}^{\ast}\Lambda_{w}C)x,y_{w}\rangle d\mu(w)\\
		&=\int_{\Omega}\langle x,(C^{'}\Lambda_{w}^{\ast}\Lambda_{w}C)y_{w}\rangle d\mu(w)\\
		&=\left\langle x,\int_{\Omega}(C^{'}\Lambda_{w}^{\ast}\Lambda_{w}C)y_{w}d\mu(w)\right\rangle.
	\end{align*}
	Then $T$ is adjointable and $T^{\ast}y=\int_{\Omega}(C^{'}\Lambda_{w}^{\ast}\Lambda_{w}C)y_{w}d\mu(w).$
	By \eqref{1}, we have $\|Tx\|^{2}\leq\|B\|^{2}\|x\|^{2}$ and so $\|T\|\leq\|B\|$, and by \eqref{1}, we have $\|Tx\|\geq\|A^{-1}\|^{-1}\|x\|$ for all  $ x\in U$ and so by Lemma \ref{3}, $T^{\ast}$ is surjective.
	This completes the proof.
\end{proof}
\begin{definition}
	Let $\{\Lambda_{w}\in End_{\mathcal{A}}^{\ast}(U,V_{w}): w\in\Omega\}$ is called a $(C$-$C^{'})$  -controlled continuous $\ast$-$g$-frame for $U$. Define the $(C$-$C^{'})$  -controlled continuous $\ast$-$g$-frame operator $S$ on $U$ by $Sx=T^{\ast}Tx=\int_{\Omega}C^{'}\Lambda_{w}^{\ast}\Lambda_{w}Cxd\mu(w)$, where $T$ is the $(C$-$C^{'})$  -controlled continuous $\ast$-$g$-frame transform.
\end{definition}

\begin{theorem}
	A  $(C$-$C^{'})$ -controlled continuous $\ast$-$g$-frame operator $S$ is  bounded, positive, self-adjoint, invertible and $\|A^{-1}\|^{-2}\leq\|S\|\leq\|B\|^{2}$.
\end{theorem}

\begin{proof}
	First,  we show that  $S$ is a self-adjoint operator. By definition,  we have,  for all $ x, y\in U$,
	\begin{align*}
		\langle Sx,y\rangle&=\left\langle\int_{\Omega}C^{'}\Lambda_{w}^{\ast}\Lambda_{w}Cxd\mu(w),y\right\rangle\\
		&=\int_{\Omega}\langle C^{'}\Lambda_{w}^{\ast}\Lambda_{w}Cx,y\rangle d\mu(w)\\
		&=\int_{\Omega}\langle x,C\Lambda_{w}^{\ast}\Lambda_{w}C^{'}y\rangle d\mu(w)\\
		&=\left\langle x,\int_{\Omega}C^{'}\Lambda_{w}^{\ast}\Lambda_{w}Cyd\mu(w)\right\rangle\\
		&=\langle x,Sy\rangle.
	\end{align*}
	Thus  $S$ is self-adjoint.\\
	By Lemma \ref{3} and Theorem \ref{2.3}, $S$ is invertible. Clearly $S$ is positive.
	By definition of a continuous $\ast$-$g$-frame,  we have
	\begin{equation*}
		A\langle x,x\rangle A^{\ast}\leq\int_{\Omega}\langle\Lambda_{w}Cx,\Lambda_{w}C^{'}x\rangle d\mu(w)\leq B\langle x,x\rangle B^{\ast}.
	\end{equation*}
	So
	\begin{equation*}
		A\langle x,x\rangle A^{\ast}\leq\langle Sx,x\rangle\leq B\langle x,x\rangle B^{\ast}.
	\end{equation*}
	This gives
	\begin{equation*}
		\|A^{-1}\|^{-2}\|x\|^{2}\leq\|\langle Sx,x\rangle\|\leq\|B\|^{2}\|x\|^{2}, \forall x\in U.
	\end{equation*}
	If we take supremum on all $x\in U$ with  $\|x\|\leq1$, then $\|A^{-1}\|^{-2}\leq\|S\|\leq\|B\|^{2}$.
\end{proof}
	\begin{theorem}\label{t3}
		Let $C \in GL^{+}(U)$, the sequence $\Lambda= \{\Lambda_{w}\in End_{\mathcal{A}}^{\ast}(U,V_{w}): w\in\Omega\}$ is a continuous $\ast$-$g$-frame for $U$ with respect to $\{V_{w}: w\in\Omega\}$ if and only if $\Lambda$ is a  $(C$-$C)$-controlled continuous $\ast$-$g$-frames for $U$ with respect to $\{V_{w}: w\in\Omega\}$.  
	\end{theorem}
	\begin{proof}
		Suppose that $\{\Lambda_{w}\}_{w\in\Omega}$ is $(C$-$C)$-controlled continuous $\ast$-$g$-frames with bounds $A$ and $B$, then
		\begin{equation*}
			A\langle x,x\rangle A^{\ast}\leq\int_{\Omega}\langle\Lambda_{w}Cx,\Lambda_{w}Cx\rangle d\mu(w)\leq B\langle x,x\rangle B^{\ast}, \qquad \forall x\in U.
		\end{equation*}
		For any $x\in U$, we have
		\begin{align*}
			A\langle x,x\rangle A^{\ast}&=A\langle CC^{-1}x,CC^{-1}x\rangle A^{\ast} \\
			&\leq A\|C\|^{2}\langle C^{-1}x,C^{-1}x\rangle A^{\ast}\\
			&\leq  \|C\|^{2}\int_{\Omega}\langle\Lambda_{w}CC^{-1}x,\Lambda_{w}CC^{-1}x\rangle d\mu(w)\\
			&=\|C\|^{2}\int_{\Omega}\langle\Lambda_{w}x,\Lambda_{w}x\rangle d\mu(w).
		\end{align*}
		On the one hand, we have
		\begin{equation}\label{eq1t3}
			A\|C\|^{-1}\|x\|^{2}(A\|C\|^{-1})^{\ast}\leq \bigg\|\int_{\Omega}\langle\Lambda_{w}x,\Lambda_{w}x\rangle d\mu(w)\bigg\|,
		\end{equation}
		on the other hand 
		\begin{align*}
			\int_{\Omega}\langle\Lambda_{w}x,\Lambda_{w}x\rangle d\mu(w)&=\int_{\Omega}\langle\Lambda_{w}CC^{-1}x,\Lambda_{w}CC^{-1}x\rangle d\mu(w)\\
			&\leq B\langle C^{-1}x,C^{-1}x\rangle B^{\ast} \\
			&\leq B\|C^{-1}\|^{2}\langle x,x\rangle B^{\ast}, 
		\end{align*}
		then
		\begin{equation}\label{eq2t3}
			\int_{\Omega}\langle\Lambda_{w}x,\Lambda_{w}x\rangle d\mu(w)\leq B\|C^{-1}\|\langle x,x\rangle B^{\ast} \|C^{-1}\|.
		\end{equation}
		From \eqref{eq1t3}, \eqref{eq2t3} and Theorem \ref{t1}, we conclude that $\{\Lambda_{w}\}_{w\in\Omega}$ is a continuous $\ast$-$g$-frame with bounds $A\|C\|^{-1}$ and $B\|C^{-1}\|$.\\
		Conversely, let $\{\Lambda_{w}\}_{w\in\Omega}$ be a continuous $\ast$-$g$-frame with bounds $A$ and $B$,\\ then for all $x\in U$, we have
		\begin{equation*}
			A\langle x,x\rangle A^{\ast}\leq\int_{\Omega}\langle\Lambda_{w}x,\Lambda_{w}x\rangle d\mu(w)\leq B\langle x,x\rangle B^{\ast}, \qquad \forall x\in U.
		\end{equation*}
		So, for all $x\in U$, we have $Cx\in U$, and
		\begin{equation*}
			\int_{\Omega}\langle\Lambda_{w}Cx,\Lambda_{w}Cx\rangle d\mu(w)\leq B\langle Cx,Cx\rangle B^{\ast} \leq B\|C\|^{2}\langle x,x\rangle B^{\ast}=B\|C\|\langle x,x\rangle B^{\ast}\|C\|.
		\end{equation*}
		Also, for all $x\in U$,
		\begin{align*}
			A\langle x,x\rangle A^{\ast}&=A\langle C^{-1}Cx,C^{-1}Cx\rangle A^{\ast}\\
			&\leq A\|C^{-1}\|^{2}\langle Cx,Cx\rangle A^{\ast}\\
			&\leq \|C^{-1}\|^{2}\int_{\Omega}\langle\Lambda_{w}Cx,\Lambda_{w}Cx\rangle d\mu(w).
		\end{align*}
		
		Hence $\Lambda$ is a  $(C$-$C)$-controlled continuous $\ast$-$g$-frames with bounds  $A\|C^{-1}\|^{-1}$ and $B\|C\|$.
		
	\end{proof}
	Let $\Lambda= \{\Lambda_{w}\in End_{\mathcal{A}}^{\ast}(U,V_{w}): w\in\Omega\}$ be a $(C$-$C^{'})$-controlled continuous $\ast$-$g$-Bessel family for $U$ with respect to $\{V_{w}: w\in\Omega\}$.\\
	The bounded linear operator  $T_{CC^{'}}:l^{2}(\{V_{w}\}_{w\in\Omega}) \rightarrow U$ given by
	\begin{equation*}
		T_{CC^{'}}(\{y_{w}\}_{w\in\Omega})=\int_{\Omega}(CC^{'})^{\frac {1}{2}}\Lambda^{\ast}_{\omega}y_{\omega}d\mu(w) \qquad  \forall \{y_{w}\}_{w\in\Omega} \in l^{2}(\{V_{w}\}_{w\in\Omega})
	\end{equation*}
	is called the synthesis operator for the  $(C$-$C^{'})$-controlled continuous $\ast$-$g$-frame $ \{\Lambda_{w}\}_{ w\in\Omega}$.\\
	The adjoint operator $T^{\ast}_{CC^{'}}: U\rightarrow l^{2}(\{V_{w}\}_{w\in\Omega})$ given by
	\begin{equation}\label{2..3}
		T^{\ast}_{CC^{'}}(x)=\{\Lambda_{\omega}(C^{'}C)^{\frac{1}{2}}x\}_{\omega \in \Omega}, \qquad \forall x\in U,
	\end{equation}
	is called the analysis operator for the  $(C$-$C^{'})$-controlled continuous $\ast$-$g$-frame $ \{\Lambda_{w} w\in\Omega\}$.\\
	When $C$ and $C^{'}$ commute with each other, and commute with the operator $\Lambda^{\ast}_{\omega}\Lambda_{\omega}$ for each $\omega \in \Omega$, then the $(C$-$C^{'})$-controlled continuous $\ast$-$g$-frames operator.
		\begin{theorem}\label{t2}
		Let $\{\Lambda_{w}\in End_{\mathcal{A}}^{\ast}(U,V_{w}): w\in\Omega\}$ and $\int_{\Omega}\langle\Lambda_{w}Cx,\Lambda_{w}C^{'}x\rangle d\mu(w) $ converge in norm, then $\{\Lambda_{w}\}_{\omega \in \Omega}$ is $(C$-$C^{'})$-controlled continuous $\ast$-$g$-frames for $U$ with respect to $\{V_{w}: w\in\Omega\}$ if and only if there exist a positive constants $A$ and $B$ such that
		\begin{equation} \label{eq1t2}
			\|A^{-1}\|^{-2}\|\langle x,x\rangle\| \leq\bigg\| \int_{\Omega}\langle\Lambda_{w}Cx,\Lambda_{w}C^{'}x\rangle d\mu(w)\bigg\| \leq \|B\|^{2}\|\langle x,x\rangle\|,  \qquad  \forall x\in U.
		\end{equation}
	\end{theorem}
	\begin{proof}
		$\Longrightarrow$)
		By the definition of  controlled continuous $\ast$-$g$-frame, we have 
		$$\langle x,x\rangle \leq A^{-1}\langle S_{CC^{'}}x,x\rangle (A^{\ast})^{-1} 
		\text{ and } 
		\langle S_{CC^{'}}x,x\rangle \leq B\langle x,x\rangle B^{\ast}$$
		Hence 
		\begin{equation*}
			\|A^{-1}\|^{-2}\|\langle x,x\rangle \| \leq\bigg\| \int_{\Omega}\langle\Lambda_{w}Cx,\Lambda_{w}C^{'}x\rangle d\mu(w)\bigg\| \leq \|B^{2}\|\|\langle x,x\rangle \|.
		\end{equation*}
		Conversely, suppose that \eqref{eq1t2} holds,we have
		\begin{equation}\label{eq2t2}
			\langle S^{\frac{1}{2}}_{CC^{'}}x,S^{\frac{1}{2}}_{CC^{'}}x\rangle=\langle S_{CC^{'}}x,x\rangle=\int_{\Omega}\langle\Lambda_{w}Cx,\Lambda_{w}C^{'}x\rangle d\mu(w)
		\end{equation} 
		Using inequality \eqref{eq2t2} in \eqref{eq1t2}, we obtaint
		then
		\begin{equation*} 
			\|A^{-1}\|^{-2}\|\langle x,x\rangle\| \leq \|\langle S^{\frac{1}{2}}_{CC^{'}}x,S^{\frac{1}{2}}_{CC^{'}}x\rangle \| \leq \|B\|^{2}\|\langle x,x\rangle\|,
		\end{equation*}
		\begin{equation*}
			\|A^{-1}\|^{-2}\|\langle x,x\rangle\| \leq \| S^{\frac{1}{2}}_{CC^{'}}x \|^{2} \leq \|B\|^{2}\|\langle x,x\rangle\|.
		\end{equation*}
		Since 
		\begin{equation}\label{eq3t2}
			\|A^{-1}\|^{-2}\|\langle x,x\rangle\| \leq \| S^{\frac{1}{2}}_{CC^{'}}x \| \leq \|B\|^{2}\|\langle x,x\rangle\|,
		\end{equation}
		from inequality \eqref{eq3t2} and Lemma \ref{l2} we conclude $\Lambda$ is a $(C$-$C^{'})$-controlled continuous $\ast$-$g$-frames for $U$ with respect to $\{V_{w}, w\in\Omega\}$.
	\end{proof}
	\begin{theorem}\label{sspp}
		Let $\{\Lambda_{w}, w\in\Omega\} \subset End^{*}_{A}(U,V_{\omega})$ and let $C,C^{'} \in GL^{+}(U)$ so that $C,C^{'}$ commute with each other and commute with $\Lambda^{\ast}_{\omega}\Lambda_{\omega}$ for all $\omega \in \Omega$. Then the following are equivalent
		\begin{itemize}
			\item [(1)] the sequence $\{\Lambda_{w}, w\in\Omega\}$ is a $(C$-$C^{'})$-controlled continuous $\ast$-$g$-Bessel sequence for $U$ with respect $\{V_{\omega}\}_{\omega \in \Omega}$ with bounds $A$ and $B$,
			\item[(2)] The operator $T_{CC^{'}}:l^{2}(\{V_{w}\}_{w\in\Omega}) \rightarrow U$ given by
			\begin{equation*}
				T_{CC^{'}}(\{y_{w}\}_{w\in\Omega})=\int_{w\in\Omega}(CC^{'})^{\frac {1}{2}}\Lambda^{\ast}_{\omega}y_{\omega}d\mu(w), \qquad \forall \{y_{w}\}_{w\in\Omega} \in l^{2}(\{V_{w}\}_{w\in\Omega})
			\end{equation*}
			is well defined and bounded operator with $\|T_{CC^{'}}\|\leq \sqrt{B}$.
		\end{itemize}
	\end{theorem}
	\begin{proof}
		$(1)\Longrightarrow (2)$\\
		Let  $\{\Lambda_{w}, w\in\Omega\}$ be a $(C$-$C^{'})$-controlled continuous $\ast$-$g$-Bessel sequence for $U$ with respect $\{V_{\omega}\}_{\omega \in \Omega}$ with bound $B$.\\
		From Theorem \ref{t2}, we have
		\begin{equation} \label{2.12}
			\bigg\|\int_{\Omega}\langle\Lambda_{w}Cx,\Lambda_{w}C^{'}x\rangle d\mu(w)\bigg\| \leq B\|x\|^{2},  \qquad \forall x\in U.
		\end{equation}
		For any sequence $\{y_{w}\}_{w\in \Omega} \in l^{2}(\{V_{\omega}\}_{\omega \in \Omega})$
		\begin{align*}
			\|T_{CC^{'}}(\{y_{w}\}_{w\in\Omega})\|^{2}&=\underset{x \in U, \|x\|=1}{\sup}\|\langle T_{CC^{'}}(\{y_{w}\}_{w\in\Omega}),x\rangle \|^{2}\\
			&=\underset{x \in U, \|x\|=1}{\sup}\bigg\|\langle \int_{\Omega}(CC^{'})^{\frac {1}{2}}\Lambda^{\ast}_{\omega}y_{\omega}d\mu(w),x\rangle\bigg\|^{2}\\
			&=\underset{x \in U, \|x\|=1}{\sup}\bigg\| \int_{\Omega}\langle(CC^{'})^{\frac {1}{2}}\Lambda^{\ast}_{\omega}y_{\omega},x\rangle d\mu(w)\bigg\|^{2}\\ 
			&=\underset{x \in U, \|x\|=1}{\sup}\bigg\| \int_{\Omega}\langle y_{\omega},\Lambda_{\omega}(CC^{'})^{\frac {1}{2}}x\rangle d\mu(w)\bigg\|^{2}\\ 
			&\leq \underset{x \in U, \|x\|=1}{\sup}\bigg\| \int_{\Omega}\langle y_{\omega},y_{\omega}\rangle d\mu(w)\bigg\| \bigg\|\int_{\Omega}\langle \Lambda_{\omega}(CC^{'})^{\frac {1}{2}}x,\Lambda_{\omega}(CC^{'})^{\frac {1}{2}}x\rangle d\mu(w)\bigg\|\\ 
			&= \underset{x \in U, \|x\|=1}{\sup}\bigg\| \int_{\Omega}\langle y_{\omega},y_{\omega}\rangle d\mu(w)\bigg\| \bigg\|\int_{\Omega}\langle \Lambda_{\omega}Cx,\Lambda_{\omega}C^{'}x\rangle d\mu(w)\bigg\|\\ 
			&\leq \underset{x \in U, \|x\|=1}{\sup}\bigg\| \int_{\Omega}\langle y_{\omega},y_{\omega}\rangle d\mu(w)\bigg\|B\|x\|^{2} =B\|\{y_{\omega}\}_{\omega \in \Omega}\|^{2}.
		\end{align*}
		Then, we have 
		\begin{align*}
			\|T_{CC^{'}}(\{y_{w}\}_{w\in\Omega})\|^{2}\leq B\|\{y_{\omega}\}_{\omega \in \Omega}\|^{2} \Longrightarrow \|T_{CC^{'}}\|\leq \sqrt{B} ,
		\end{align*}
		we conclude the operator $T_{CC^{'}}$ is well defined and bounded.\\
		$(2)\Longrightarrow (1)$\\
		Let the operator $T_{CC^{'}}$ is well defined, bounded and $\|T_{CC^{'}}\|\leq \sqrt{B}$ .\\
		For any $x\in U$ and finite subset $\Psi \subset \Omega$, we have
		\begin{align*}
			\int_{\Psi}\langle\Lambda_{w}Cx,\Lambda_{w}C^{'}x\rangle d\mu(w)&=\int_{\Psi}\langle C^{'}\Lambda^{\ast}_{w}\Lambda_{w}Cx,x\rangle d\mu(w)\\
			&=\int_{\Psi}\langle (CC^{'})^{\frac {1}{2}}\Lambda^{\ast}_{w}\Lambda_{w}(CC^{'})^{\frac {1}{2}}x,x\rangle d\mu(w)\\
			&=\langle T_{CC^{'}}(\{y_{w}\}_{w\in\Psi}),x\rangle\\
			&\leq \| T_{CC^{'}}\|\|(\{y_{w}\}_{w\in\Psi})\|\|x\|
		\end{align*}
		where  $y_{w}=\Lambda_{w}(CC^{'})^{\frac {1}{2}}x$ if $\omega \in \Psi$ and $y_{w}=0$ if $\omega \notin \Psi$.\\
		Therefore, 
		\begin{align*}
			\int_{\Psi}\langle\Lambda_{w}Cx,\Lambda_{w}C^{'}x\rangle d\mu(w) &\leq \| T_{CC^{'}}\|(\int_{\Psi}\|\Lambda_{w}(CC^{'})^{\frac {1}{2}}x\|^{2}d\mu(w))^{\frac {1}{2}}\|x\|\\
			&=\| T_{CC^{'}}\|(\int_{\Psi}\langle\Lambda_{w}Cx,\Lambda_{w}C^{'}x\rangle d\mu(w))^{\frac {1}{2}}\|x\|
		\end{align*}
		Since $\Psi$ is arbitrary, we have
		\begin{align*}
			\int_{\Omega}\langle\Lambda_{w}Cx,\Lambda_{w}C^{'}x\rangle d\mu(w) &\leq \| T_{CC^{'}}\|^{2}\|x\|^{2}
		\end{align*}
		\begin{align*}
			\Longrightarrow \int_{\Omega}\langle\Lambda_{w}Cx,\Lambda_{w}C^{'}x\rangle d\mu(w) &\leq B\|x\|^{2} \qquad as  \| T_{CC^{'}}\|\leq \sqrt{B}
		\end{align*}
		Therfore $\{\Lambda_{w}, w\in\Omega\}$ is a $(C$-$C^{'})$-controlled continue $\ast$-$g$-Bessel sequence for $U$ with respect to $\{V_{\omega}\}_{\omega \in \Omega}$.
	\end{proof}

	From now, we assume that $C$ and $C^{'}$ commute with each other, and commute with the operator $\Lambda^{\ast}_{\omega}\Lambda_{\omega}$ for each $\omega \in \Omega$.
	\begin{proposition}

		Let $\Lambda= \{\Lambda_{w}\in End_{\mathcal{A}}^{\ast}(U,V_{w}): w\in\Omega\}$ be a $(C$-$C^{'})$-controlled continuous $\ast$-$g$-frame for Hilbert $C^{\ast}$-module $U$ and let $S_{CC^{'}}:U\longrightarrow U$ defined by
		$S_{CC^{'}}x=T_{CC^{'}}T^{\ast}_{CC^{'}}x=\int_{\Omega}C^{'}\Lambda^{\ast}_{w}\Lambda_{w}Cx d\mu(w)$. The operator $S_{CC^{'}}$ called the $(C$-$C^{'})$-controlled continuous $\ast$-$g$-frames operator 
		is bounded, positive, sefladjoint and invertible.
	\end{proposition}
\begin{proof}
	By the definition of $(C$-$C^{'})$-controlled continuous $\ast$-$g$-frames operator $
		S_{CC^{'}}$, we have
		
		\begin{equation*}
			A\langle x,x\rangle A^{\ast}\leq \langle S_{CC^{'}}x,x\rangle \leq B\langle x,x\rangle B^{\ast}
		\end{equation*} 
		so
		\begin{equation*}
			A.Id_{U}A^{\ast} \leq S_{CC^{'}} \leq B.Id_{U}B^{\ast},
		\end{equation*} 
		where $Id_{U}$ is the identity operator in $U$.\\
		It is clear that $S_{CC^{'}}$ is a positive operator.\\
		Thus the $(C$-$C^{'})$-controlled continuous $\ast$-$g$-frames operator  $S_{CC^{'}}$is bounded and invertible.
		In other hand we know every positive operator is self-adjoint.
	\end{proof}
\begin{theorem}
	Let $\{\Lambda_{w}\}_{w\in \Omega}$ be a $(C$-$C^{'})$-controlled continuous $\ast$-$g$-frame for $U$ with $(C$-$C^{'})$-controlled continuous $\ast$-$g$-frame transform $T$. Then $\{\Lambda_{w}\}_{w\in \Omega}$ is a $(C$-$C^{'})$-controlled continuous $g$-frame for $U$ with lower and upper frame bounds $||(T^{\ast}T)^{-1}||^{-1}$ and $||T||^{2}$, respectively.
\end{theorem}

\begin{proof}
	By Theorem \ref{2.3}, $T$ is injective and has a closed range, and so  by Lemma \ref{3}, 
	\begin{equation*}
		||(T^{\ast}T)^{-1}||^{-1}\langle x,x\rangle\leq \langle T^{\ast}Tx,x\rangle\leq ||T||^{2}\langle x,x\rangle,\qquad\forall x\in U.    
	\end{equation*}
	So 
	\begin{equation*}
		||(T^{\ast}T)^{-1}||^{-1}\langle x,x\rangle\leq \int_{\Omega}\langle \Lambda_{w}Cx,\Lambda_{w}C^{'}x\rangle d\mu(w)\leq ||T||^{2}\langle x,x\rangle,\qquad\forall x\in U.     
	\end{equation*}
	Hence $\{\Lambda_{w}\}_{w\in \Omega}$ is a $(C$-$C^{'})$-controlled continuous $g$-frame for $U$ with lower and upper frame bounds $||(T^{\ast}T)^{-1}||^{-1}$ and $||T||^{2}$, respectively.
\end{proof}

\begin{theorem}
	Let $\{\Lambda_{w}\}_{w\in\Omega}$ and $\{\Gamma_{w}\}_{w\in\Omega}$ be $(C$-$C^{'})$-controlled continuous $\ast$-$g$-Bessel sequences for Hilbert $C^{\ast}$-modules $U_{1}$ and $U_{2}$ with $(C$-$C^{'})$-controlled continuous $\ast$-$g$-Bessel bounds $B_{1}$ and $B_{2}$, respectively. Then $\{\Lambda_{w}^{\ast}\Gamma_{w}\}_{w\in\Omega}$ is a $(C$-$C^{'})$-controlled continuous $\ast$-$g$-Bessel sequence for $U_{2}$ with respect to $U_{1}$.
\end{theorem}

\begin{proof}
	We have for each $x\in U_{2}$,
	\begin{align*}
		\int_{\Omega}\langle \Lambda_{w}^{\ast}\Gamma_{w}Cx,\Lambda_{w}^{\ast}\Gamma_{w}C^{'}x\rangle d\mu(w)&\leq \int_{\Omega}||\Lambda_{w}^{\ast}||^{2}\langle \Gamma_{w}Cx,\Gamma_{w}C^{'}x\rangle d\mu(w) \\&\leq ||B_{1}||^{2}\int_{\Omega}\langle \Gamma_{w}Cx,\Gamma_{w}C^{'}x\rangle d\mu(w)\\&\leq ||B_{1}||^{2}B_{2}\langle x,x\rangle B_{2}^{\ast}\\&\leq  ||B_{1}||B_{2}\langle x,x\rangle(||B_{1}||B_{2})^{\ast} .
	\end{align*}
	Hence $\{\Lambda_{w}^{\ast}\Gamma_{w}\}_{w\in\Omega}$ is a $(C$-$C^{'})$-controlled continuous $\ast$-$g$-Bessel sequence for $U_{2}$ with respect to $U_{1}$.
\end{proof}

\begin{theorem}\label{th}
	Let	$\{\Lambda_{w}\in End_{\mathcal{A}}^{\ast}(U,V_{w}): w\in\Omega\}$  be a $(C$-$C^{'})$-controlled continuous $\ast$-$g$-frame for Hilbert $C^{\ast}$-module $U$.  If the operator $\theta:\oplus_{w\in\Omega}V_{w}\rightarrow U$,  defined by $\theta(\{x_{w}\}_{w\in \Omega})=\int_{\Omega}\Lambda_{w}^{\ast}x_{w}d\mu(w)$,  is surjective, then  $\{\Lambda_{w}\}_{w\in \Omega}$ is a $(C$-$C^{'})$-controlled continuous $\ast$-$g$-frame for $U$.
\end{theorem}

\begin{proof}
	For each $x\in U$, 
	\begin{align*}
		\left\|\int_{\Omega}\langle \Lambda_{w}Cx,\Lambda_{w}C^{'}x\rangle d\mu(w)\right\|&=\left\|\int_{\Omega}\langle x,C\Lambda_{w}^{\ast}\Lambda_{w}C^{'}x\rangle d\mu(w)\right\|\\&=\left\|\langle x,\int_{\Omega}C\Lambda_{w}^{\ast}\Lambda_{w}C^{'}x d\mu(w)\rangle\right\|\\&\leq \|x\| \;\left\|\int_{\Omega}C\Lambda_{w}^{\ast}\Lambda_{w}C^{'}x d\mu(w)\right\|\\&\leq \|x\|\;\left\|\theta(\{\Lambda_{w}^{\ast}\Lambda_{w}x\}_{w\in \Omega})\right\|\\&\leq \|x\|\;\|\theta\|\;\left\|\{\Lambda_{w}x\}_{w\in \Omega}\right\|\\&\leq\|x\|\;\|\theta\|\;\left\|\int_{\Omega}\langle \Lambda_{w}Cx,\Lambda_{w}C^{'}x \rangle d\mu(w)\right\|^{\frac{1}{2}}.
	\end{align*}
	Thus  
	\begin{equation*}
		\left\|\int_{\Omega}\langle \Lambda_{w}Cx,\Lambda_{w}C^{'}x\rangle d\mu(w)\right\|^{\frac{1}{2}}\leq\|\theta\|\;\|x\|.
	\end{equation*}
	So 
	\begin{equation}\label{eq8}
		\left\|\int_{\Omega}\langle \Lambda_{w}Cx,\Lambda_{w}C^{'}x\rangle d\mu(w)\right\|\leq \|\theta\|^{2}\;\|x\|^{2},\qquad \forall x\in U.
	\end{equation}
	Since $\theta$ is  surjective, by Lemma \ref{l2}, there exists $\nu >0$ such that 
	\begin{equation*}
		||\theta^{\ast}x||\geq \nu ||x||,\qquad \forall x\in U.
	\end{equation*}
	Therefore, $\theta^{\ast}$ is injective.  \\
	Hence $\theta^{\ast}: U\rightarrow \mathcal{R}(\theta^{\ast})$ is invertible, and  for each $x\in U$,
	$(\theta^{\ast}_{/\mathcal{R}(\theta^{\ast})})^{-1}\theta^{\ast}x=x$.\\	
	So, for each $x\in U$, 
	$$
	\|x\|=\|(\theta^{\ast}_{/\mathcal{R}(\theta^{\ast})})^{-1}\theta^{\ast}x\|\leq \|(\theta^{\ast}_{/\mathcal{R}(\theta^{\ast})})^{-1}\|\;\|\theta^{\ast}x\|.  
	$$
	Thus 
	\begin{equation}\label{eq12}
		\|(\theta^{\ast}_{/\mathcal{R}(\theta^{\ast})})^{-1}\|^{-2}\;\|x\|^{2}\leq \left\|\int_{\Omega}\langle \Lambda_{w}x,\Lambda_{w}x\rangle d\mu(w)\right\|.
	\end{equation}
	From \eqref{eq8} and \eqref{eq12}, $\{\Lambda_{w}\}_{w\in \Omega}$ is a  $(C$-$C^{'})$-controlled continuous $\ast$-$g$-frame for $U$.
\end{proof}

\begin{theorem}
	Let $\{\Lambda_{w}\}_{w\in \Omega}$ be a  $(C$-$C^{'})$-controlled continuous $\ast$-$g$-frame for $U$. If $\{\Gamma_{w}\}_{w\in \Omega}$  is  a  $(C$-$C^{'})$-controlled continuous $\ast$-$g$-Bessel sequence for $U$ with respect to $\{V_{w}: {w\in \Omega}\}$, and the operator $F:U\rightarrow U$, defined by $Fx=\int_{\Omega}\Gamma_{w}^{\ast}\Lambda_{w}xd\mu(w)$, is surjective, then $\{\Gamma_{w}\}_{w\in \Omega}$ is a  $(C$-$C^{'})$-controlled continuous $\ast$-$g$-frame for $U$.
\end{theorem}
\begin{proof}
	Since $\{\Lambda_{w}\}_{w\in \Omega}$ is  a continuous $\ast$-$g$-frame for $U$, we have a continuous $\ast$-$g$-frame transform $T:U\rightarrow \oplus_{w\in\Omega}V_{w}$, defined by $Tx=\{\Lambda_{w}x\}_{w\in \Omega}$.\\	
	Now  the operator $K:\oplus_{w\in\Omega}V_{w}\rightarrow U$, defined  by $K(\{x_{w}\}_{w\in \Omega})=\int_{\Omega}\Gamma_{w}^{\ast}x_{w}d\mu(w)$,  is well-defined,  since 
	\begin{align*}
		\left\|\int_{\Omega}\Gamma_{w}^{\ast}x_{w}d\mu(w)\right\|&=\sup_{\|y\|=1}\left\|\langle \int_{\Omega}\Gamma_{w}^{\ast}x_{w}d\mu(w), y\rangle\right\|\\&=\sup_{\|y\|=1}\left\|\int_{\Omega}\langle x_{w},\Gamma_{w}y\rangle d\mu(w)\right\|\\&\leq\sup_{\|y\|=1}\left\|\int_{\Omega}\langle x_{w},x_{w}\rangle d\mu(w) \right\|^{\frac{1}{2}}\left\|\int_{\Omega}\langle \Gamma_{w}y,\Gamma_{w}y\rangle d\mu(w)\right\|^{\frac{1}{2}}\\&\leq\sup_{\|y\|=1}\|\{x_{w}\}_{w\in \Omega}\|\|C\langle y,y\rangle C^{\ast}\|^{\frac{1}{2}}= \|\{x_{w}\}_{w\in \Omega}\|\|C\|.
	\end{align*}
	We have for each $x\in U$,
	\begin{equation*}
		Fx=\int_{\Omega}\Gamma_{w}^{\ast}\Lambda_{w}xd\mu(w)=KTx.
	\end{equation*}
	Hence $F=KT$. Since $F$ is surjective,  for each $x\in U$, there exists $y\in U$ such that  $Fy=x$, which  implies $x=Fy=KTy$ and $Ty\in \oplus_{w\in\Omega}V_{w}$ and so $K$ is surjective. From Theorem \ref{th},
	we conclude that $\{\Gamma_{w}\}_{w\in \Omega}$ is a continuous $\ast$-$g$-frame for $U$.
\end{proof}

In the following we study continuous $\ast$-$g$-frames in two Hilbert $C^{\ast}$-modules with different $C^{\ast}$-algebras.

\begin{theorem}
	Let $(U,\mathcal{A}, \langle .,.\rangle_{\mathcal{A}})$ and $(U,\mathcal{B}, \langle .,.\rangle_{\mathcal{B}})$ be two Hilbert $C^{\ast}$-modules,  $\phi :\mathcal{A} \rightarrow \mathcal{B}$ be a $\ast$-homomorphism and $\theta$ be an adjointable map on $U$ such that $\langle \theta x,\theta y\rangle_{\mathcal{B}}=\phi(\langle x,y\rangle_{\mathcal{A}})$ for all $x, y\in U$. Also, suppose that $\{\Lambda_{w}\}_{w\in\Omega}$ is a  $(C$-$C^{'})$-controlled continuous $\ast$-$g$-frame for $(U, \mathcal{A},\langle .,.\rangle_{\mathcal{A}})$ with  $(C$-$C^{'})$-controlled continuous $\ast$-$g$-frame operator $S_{\mathcal{A}}$ and lower and upper bounds $A$, $B$ respectively. If $\theta$ is surjective and $\theta\Lambda_{w}=\Lambda_{w}\theta$ for all $w\in\Omega$, then $\{\Lambda_{w}\}_{w\in\Omega}$ is a  $(C$-$C^{'})$-controlled continuous $\ast$-$g$-frame for $(U,\mathcal{B}, \langle .,.\rangle_{\mathcal{B}})$ with  $(C$-$C^{'})$-controlled continuous $\ast$-$g$-frame operator $S_{\mathcal{B}}$ and lower and upper bounds $\phi(A)$ and $\phi(B)$, respectively, and $\langle S_{\mathcal{B}}\theta x,\theta y\rangle_{\mathcal{B}} =\phi(\langle S_{\mathcal{A}}x, y\rangle_{\mathcal{A}}).$ 
\end{theorem}

\begin{proof}
	Let $y\in U$. Since $\theta$ is surjective,  there exists $x\in U$ such that $\theta x=y$, and we have 
	\begin{equation*}
		A\langle x,x\rangle_{\mathcal{A}} A^{\ast}\leq \int_{\Omega}\langle \Lambda_{w}Cx,\Lambda_{w}C^{'}x\rangle_{\mathcal{A}} d\mu(w)\leq B\langle x,x\rangle_{\mathcal{A}} B^{\ast}.
	\end{equation*}
	Thus 
	\begin{equation*}
		\phi(	A\langle x,x\rangle_{\mathcal{A}} A^{\ast})\leq \phi\big( \int_{\Omega}\langle \Lambda_{w}Cx,\Lambda_{w}C^{'}x\rangle_{\mathcal{A}} d\mu(w)\big)\leq \phi(B\langle x,x\rangle_{\mathcal{A}} B^{\ast}).
	\end{equation*}
	By definition of $\ast$-homomorphism, we have 
	\begin{equation*}
		\phi(A)\phi(\langle x,x\rangle_{\mathcal{A}} )\phi(A^{\ast})\leq \int_{\Omega}\phi\big(\langle \Lambda_{w}Cx,\Lambda_{w}C^{'}x\rangle_{\mathcal{A}}\big)d\mu(w)\leq\phi(B)\phi(\langle x,x\rangle_{\mathcal{A}})\phi(B^{\ast}).
	\end{equation*}
	By the relation betwen $\theta$ and $\phi$, we get 
	\begin{equation*}
		\phi(A)\langle y,y\rangle_{\mathcal{B}}\phi(A)^{\ast}\leq \int_{\Omega}\langle \Lambda_{w}Cy,\Lambda_{w}C^{'}y\rangle_{\mathcal{B}}d\mu(w)\leq  \phi(B)\langle y,y\rangle_{\mathcal{B}}\phi(B)^{\ast}.
	\end{equation*}
	On the other hand, we have
	\begin{align*}
		\phi(\langle S_{\mathcal{A}}x, y\rangle_{\mathcal{A}})&=\phi(\langle \int_{\Omega}C^{'}\Lambda_{w}^{\ast}\Lambda_{w}Cxd\mu(w),y\rangle_{\mathcal{A}})\\&=\int_{\Omega}\phi(\langle \Lambda_{w}Cx,\Lambda_{w}C^{'}y\rangle_{\mathcal{A}})d\mu(w)\\&=\int_{\Omega}\langle \Lambda_{w}\theta C x,\Lambda_{w}\theta C^{'}y\rangle_{\mathcal{B}}d\mu(w)\\&=\langle \int_{\Omega}C\Lambda_{w}^{\ast}\Lambda_{w}\theta C^{'} xd\mu(w), \theta y\rangle_{\mathcal{B}}\\&=\langle S_{\mathcal{B}}\theta x,\theta y\rangle_{\mathcal{B}}.
	\end{align*}
	This completes the proof.
\end{proof}

\begin{theorem}
	Let $\{\Lambda_{w}\in End_{\mathcal{A}}^{\ast}(U,V_{w}): w\in\Omega\}$ be a $(C$-$C^{'})$-controlled continuous $\ast$-$g$-frame for $U$ with lower and upper bounds $A$ and $B$, respectively. Let $\theta\in End_{\mathcal{A}}^{\ast}(U)$ be injective and have a closed range. Then $\{\theta \Lambda_{w}\}_{w\in \Omega}$ is a $(C$-$C^{'})$-controlled continuous $\ast$-$g$-frame for $U$.
\end{theorem}

\begin{proof}
	We have 
	\begin{equation*}
		A\langle x,x\rangle A^{\ast}\leq \int_{\Omega}\langle \Lambda_{w}Cx,\Lambda_{w}C^{'}x\rangle d\mu(w)\leq B\langle x,x\rangle B^{\ast},\qquad\forall x\in U.
	\end{equation*}
	Then for each $x\in U$ 
	\begin{equation}\label{eq13}
		\int_{\Omega}\langle \theta\Lambda_{w}Cx,\theta\Lambda_{w}C^{'}x\rangle d\mu(w)\leq ||\theta ||^{2}B\langle x,x\rangle B^{\ast}\leq (||\theta ||B)\langle x,x\rangle (||\theta||B)^{\ast}.
	\end{equation}
	By Lemma \ref{3}, we have for each $x\in U$
	\begin{equation*}
		||(\theta^{\ast}\theta)^{-1}||^{-1}\langle \Lambda_{w}Cx,\Lambda_{w}C^{'}x\rangle \leq \langle \theta\Lambda_{w}Cx,\theta\Lambda_{w}C^{'}x\rangle
	\end{equation*}
	and $||\theta^{-1}||^{-2}\leq ||(\theta^{\ast}\theta)^{-1}||^{-1}$. Thus  
	\begin{equation}\label{eq14}
		||\theta^{-1}||^{-1}A\langle x,x\rangle (||\theta^{-1}||^{-1}A)^{\ast}\leq \int_{\Omega}\langle \theta\Lambda_{w}Cx,\theta\Lambda_{w}C^{'}x\rangle d\mu(w).
	\end{equation}
	From \eqref{eq13} and \eqref{eq14}, we have for each $x\in U$
	\begin{eqnarray*}
		||\theta^{-1}||^{-1}A\langle x,x\rangle (||\theta^{-1}||^{-1}A)^{\ast} & \leq & \int_{\Omega}\langle \theta\Lambda_{w}Cx,\theta\Lambda_{w}C^{'}x\rangle d\mu(w)\\ & \leq &  ||\theta ||^{2}B\langle x,x\rangle B^{\ast}\\ &\leq & (||\theta ||B)\langle x,x\rangle (||\theta||B)^{\ast}.
	\end{eqnarray*}
	We conclude that $\{\theta\Lambda_{w}\}_{w\in \Omega}$ is a $(C$-$C^{'})$-controlled continuous $\ast$-$g$-frame for $U$.
\end{proof}

\begin{theorem}\label{18}
	Let $\{\Lambda_{w}\in End_{\mathcal{A}}^{\ast}(U,V_{w}): w\in\Omega\}$ be a $(C$-$C^{'})$-controlled continuous $\ast$-$g$-frame for $U$ with lower and upper bounds $A$ and $B$, respectively, and with $(C$-$C^{'})$-controlled continuous $\ast$-$g$-frame operator $S$. Let $\theta\in End_{\mathcal{A}}^{\ast}(U)$ be injective and have a closed range. Then $\{\Lambda_{w}\theta: w\in\Omega\}$ is a $(C$-$C^{'})$-controlled continuous $\ast$-$g$-frame for $U$ with $(C$-$C^{'})$-controlled continuous $\ast$-$g$-frame operator $\theta^{\ast}S\theta$.
\end{theorem}

\begin{proof}
	We have 
	\begin{equation}\label{eq11}
		A\langle \theta x,\theta x\rangle A^{\ast}\leq\int_{\Omega}\langle\Lambda_{w}C\theta x,\Lambda_{w}C^{'}\theta x\rangle d\mu(w)\leq B\langle \theta x,\theta x\rangle B^{\ast}, \forall x\in U.
	\end{equation}
	Using Lemma \ref{3}, we have $\|(\theta^{\ast}\theta)^{-1}\|^{-1}\langle x,x\rangle\leq\langle \theta x,\theta x\rangle$, $\forall x\in U$. That is, $\|\theta^{-1}\|^{-2}\leq\|(\theta^{\ast}\theta)^{-1}\|^{-1}$. This implies
	\begin{equation}\label{eq22} 
		\|\theta^{-1}\|^{-1}A\langle x,x\rangle(\|\theta^{-1}\|^{-1}A)^{\ast}\leq A\langle \theta x,\theta x\rangle A^{\ast}, \forall x\in U.
	\end{equation}
	And we know that $\langle \theta x,\theta x\rangle\leq\|\theta\|^{2}\langle x,x\rangle$, $\forall x\in U$. This implies that
	\begin{equation}\label{eq33}
		B\langle \theta x,\theta x\rangle B^{\ast}\leq\|\theta \|B\langle x,x\rangle(\|\theta\|B)^{\ast}, \forall x\in U.
	\end{equation}
	Using \eqref{eq11}, \eqref{eq22} and  \eqref{eq33},  we have
	\begin{equation}
		\|\theta^{-1}\|^{-1}A\langle x,x\rangle(\|\theta^{-1}\|^{-1}A)^{\ast}\leq\int_{\Omega}\langle\Lambda_{w}C\theta x,\Lambda_{w}C^{'}\theta x\rangle d\mu(w)\leq B\|\theta \|\langle x,x\rangle(B\|\theta \|)^{\ast}, \forall x\in U.
	\end{equation}
	So $\{\Lambda_{w}\theta: w\in\Omega\}$ is a $(C$-$C^{'})$-controlled continuous $\ast$-$g$-frame for $U$.
	
	Moreover for every $x\in U$, we have
	\begin{align*}
		\theta^{\ast}S\theta x&=\theta^{\ast}\int_{\Omega}C^{'}\Lambda_{w}^{\ast}\Lambda_{w}C\theta xd\mu(w)\\
		&=\int_{\Omega}C^{'}\theta^{\ast}\Lambda_{w}^{\ast}\Lambda_{w}C\theta xd\mu(w)\\
		&=\int_{\Omega}C^{'}(\Lambda_{w}\theta)^{\ast}(\Lambda_{w}\theta)Cxd\mu(w).
	\end{align*}
	This completes the proof.	
\end{proof}

\begin{corollary}\label{01}
	Let $\{\Lambda_{w}\in End_{\mathcal{A}}^{\ast}(U,V_{w}): w\in\Omega\}$ be a $(C$-$C^{'})$-controlled continuous $\ast$-$g$-frame for $U$, with $(C$-$C^{'})$-controlled continuous $\ast$-$g$-frame operator $S$. Then $\{\Lambda_{w}S^{-1}: w\in\Omega\}$ is a $(C$-$C^{'})$-controlled continuous $\ast$-$g$-frame for $U$.
\end{corollary}

\begin{proof}
	The proof follows from Theorem \ref{18}  by taking $\theta=S^{-1}$.
\end{proof}

\section{The duality of continuous $\ast$-$g$-frames}\label{section4}

\begin{definition}
	A $(C$-$C^{'})$-controlled continuous $\ast$-$g$-frame $\{\Gamma_{w}\}_{w\in\Omega}$ is a  $(C$-$C^{'})$-controlled continuous dual $\ast$-$g$-frame for a given $(C$-$C^{'})$-controlled continuous $\ast$-$g$-frame  $\{\Lambda_{w}\}_{w\in\Omega}$ if 
	\begin{equation*}
		x=\int_{\Omega}C^{'}\Lambda_{w}^{\ast}\Gamma_{w}Cxd\mu(w),\qquad\forall x\in U.
	\end{equation*} 
	The $(C$-$C^{'})$-controlled continuous $\ast$-$g$-frame $\{\Lambda_{w}S^{-1}\}_{w\in\Omega}$ is called the canonical $(C$-$C^{'})$-controlled continuous
	dual $\ast$-$g$-frame for $\{\Lambda_{w}\}_{w\in \Omega}$.
\end{definition}

\begin{remark}
	By Corollary \ref{01}, every $(C$-$C^{'})$-controlled continuous $\ast$-$g$-frame for $U$ has a  $(C$-$C^{'})$-controlled continuous dual $\ast$-$g$-frame.
\end{remark}

\begin{definition}
	Let $\{\Lambda_{w}\}_{w\in\Omega}$ and $\{\Gamma_{w}\}_{w\in\Omega}$ be $(C$-$C^{'})$-controlled continuous $\ast$-$g$-frames for $U$. Then  two $(C$-$C^{'})$-controlled continuous $\ast$-$g$-frames are similar if there exists an adjointable operator $Q$ on $U$ such that 
	\begin{equation*}
		\Gamma_{w}=\Lambda_{w}Q, \qquad\forall w\in\Omega.
	\end{equation*}
\end{definition}

\begin{theorem}\label{22}
	Let $\{\Lambda_{w}\}_{w\in\Omega}$ be a $(C$-$C^{'})$-controlled continuous $\ast$-$g$-frame for $U$, with $(C$-$C^{'})$-controlled continuous $\ast$-$g$-frame transform $T_{\Lambda}$, and let $Q\in End^{\ast}_{\mathcal{A}}(U)$ be invertible. Then every  $(C$-$C^{'})$-controlled continuous of $(C$-$C^{'})$-controlled continuous dual $\ast$-$g$-frame for $\{\Lambda_{w}Q^{\ast}\}_{w\in\Omega}$ is similar to a  $(C$-$C^{'})$-controlled continuous  dual of $\{\Lambda_{w}\}_{w\in\Omega}$, and  the converse does  also hold. 
\end{theorem}

\begin{proof}
	Let $\{\Gamma_{w}\}_{w\in\Omega}$ be  a $(C$-$C^{'})$-controlled continuous dual of $\{\Lambda_{w}Q^{\ast}\}_{w\in\Omega}$, with $(C$-$C^{'})$-controlled continuous $\ast$-$g$-frame transform $T_{\Gamma}$. By Theorem \ref{18}, $\{\Lambda_{w}Q^{\ast}\}_{w\in\Omega}$ is a $(C$-$C^{'})$-controlled continuous $\ast$-$g$-frame for $U$ with $(C$-$C^{'})$-controlled continuous $\ast$-$g$-frame transform $T_{\Lambda Q^{\ast}}$. So  for each $x\in U$,
	\begin{eqnarray}\label{19}
		T_{\Lambda Q^{\ast}}x=\{\Lambda_{w}Q^{\ast}x\}_{w\in\Omega}=T_{\Lambda}(Q^{\ast}x).
	\end{eqnarray}
	And for each $x\in U$, we have 
	\begin{align*}
		x&=\int_{\Omega}C^{'}(\Lambda_{w}Q^{\ast})^{\ast}\Gamma_{w}Cxd\mu(w)\\&=\int_{\Omega}C^{'}Q\Lambda_{w}^{\ast}\Gamma_{w}Cxd\mu(w)\\&=Q\big(\int_{\Omega}C^{'}\Lambda_{w}^{\ast}\Gamma_{w}Cxd\mu(w)\big).
	\end{align*}
	So 
	\begin{equation*}
		C^{'}T^{\ast}_{\Lambda Q^{\ast}}T_{\Gamma}C=QC^{'}T_{\Lambda}^{\ast}T_{\Gamma}C=I_{U}.
	\end{equation*}
	By the invertibility of $Q$, we have $Q^{-1}=C^{'}T_{\Lambda}^{\ast}T_{\Gamma}C$ and  $C^{'}T_{\Lambda}^{\ast}T_{\Gamma}CQ=I_{U}$ and from \eqref{19}, $T_{\Lambda}^{\ast}T_{\Gamma Q}=I_{U}$.
	Hence $\{\Gamma_{w}Q\}_{w\in\Omega}$ is a $(C$-$C^{'})$-controlled continuous dual for $\{\Lambda_{w}\}_{w\in\Omega}$ that is similar to $\{\Gamma_{w}\}_{w\in\Omega}$.
	
	Now suppose that $\{\Gamma_{w}\}_{w\in\Omega}$ is a $(C$-$C^{'})$-controlled continuous dual of $\{\Lambda_{w}\}_{w\in\Omega}$ with $(C$-$C^{'})$-controlled continuous $\ast$-$g$-frame transform $T_{\Gamma}$. Then we have for each $x\in U$,
	\begin{align*}
		x=\int_{\Omega}C^{'}\Lambda_{w}^{\ast}\Gamma_{w}Cxd\mu(w) &\implies I_{U}=C^{'}T_{\Lambda}^{\ast}T_{\Gamma}C \\&\implies C^{'}QT_{\Lambda}^{\ast}T_{\Gamma}Q^{-1}C=I_{U}\\&\implies C^{'}T_{\Lambda Q^{\ast}}^{\ast}T_{\Gamma Q^{-1}}C=I_{U}.
	\end{align*}
	Hence $\{\Gamma_{w}Q^{-1}\}_{w\in\Omega}$ is a  $(C$-$C^{'})$-controlled continuous dual of $\{\Lambda_{w}Q^{\ast}\}_{w\in\Omega}$, which is  similar to $\{\Gamma_{w}\}_{w\in\Omega}$.
\end{proof}

\begin{theorem}
	If $\{\Lambda_{w}\}_{w\in\Omega}$ and $\{\Gamma_{w}\}_{w\in\Omega}$ are $(C$-$C^{'})$-controlled continuous $\ast$-$g$-frames with frame operators $S_{\Lambda}$ and $S_{\Gamma}$, respectively, then there exists a similar $(C$-$C^{'})$-controlled continuous $\ast$-$g$-frame to $\{\Gamma_{w}\}_{w\in\Omega}$ with frame operator $S_{\Lambda}$ and its $(C$-$C^{'})$-controlled continuous dual is      $\{\theta_{w}S_{\Gamma}^{\frac{1}{2}}S_{\Lambda}^{-\frac{1}{2}}\}_{w\in\Omega}$,  where $\{\theta_{w}\}_{w\in\Omega}$ is a $(C$-$C^{'})$-controlled continuous dual of $\{\Gamma_{w}\}_{w\in\Omega}$.
\end{theorem}

\begin{proof}
	Let $Q=S_{\Lambda}^{\frac{1}{2}}S_{\Gamma}^{-\frac{1}{2}}$. Then by Theorem \ref{18}, $\{\Gamma_{w}Q^{\ast}\}_{w\in\Omega}$ is a $(C$-$C^{'})$-controlled continuous  $\ast$-$g$-frame for $U$ and it is similar to $\{\Gamma_{w}\}_{w\in\Omega}$, where  $S_{\Gamma Q^{\ast}}=QS_{\Gamma}Q^{\ast}$ is a $(C$-$C^{'})$-controlled continuous $\ast$-$g$-frame operator  of  $\{\Gamma_{w}Q^{\ast}\}_{w\in\Omega}$, and  so we have
	\begin{align*}
		S_{\Gamma Q^{\ast}}&=	S_{\Lambda}^{\frac{1}{2}}S_{\Gamma}^{-\frac{1}{2}}S_{\Gamma}(S_{\Lambda}^{\frac{1}{2}}S_{\Gamma}^{-\frac{1}{2}})^{\ast}\\&= S_{\Lambda}^{\frac{1}{2}}S_{\Gamma}^{-\frac{1}{2}}S_{\Gamma}S_{\Gamma}^{-\frac{1}{2}}S_{\Lambda}^{\frac{1}{2}}\\&=S_{\Lambda}.
	\end{align*}
	Let $\{\alpha_{w}\}_{w\in\Omega}$ be  a $(C$-$C^{'})$-controlled continuous dual of $\{\Gamma_{w}Q^{\ast}\}_{w\in\Omega}$. Then by Theorem \ref{22}, it is similar to a $(C$-$C^{'})$-controlled continuous dual of $\{\Gamma_{w}\}_{w\in\Omega}$ and  so $\alpha_{w}=\theta_{w}Q^{\ast}$ such that $\{\theta_{w}\}_{w\in\Omega}$ is a $(C$-$C^{'})$-controlled continuous dual of $\{\Gamma_{w}\}_{w\in\Omega}$.
\end{proof}

\begin{definition}
	Let $\{\Lambda_{w}\}_{w\in\omega}$ and $\{\Gamma_{w}\}_{w\in\Omega}$ be two $(C$-$C^{'})$-controlled continuous $\ast$-$g$-frames for $U$. If there exists an invertible adjointable operator $K$ on $U$ such that 
	\begin{equation*}
		x=\int_{\Omega}C^{'}\Lambda_{w}^{\ast}\Gamma_{w}KCxd\mu(w),\qquad\forall x\in U, 
	\end{equation*} 
	then we call $\{\Gamma_{w}\}_{w\in\Omega}$ an $(C$-$C^{'})$-controlled continuous operator dual of $\{\Lambda_{w}\}_{w\in\Omega}$ with corresponding invertible operator $K$.
\end{definition}

\begin{remark}
	$\{\Gamma_{w}\}_{w\in\Omega}$ is an $(C$-$C^{'})$-controlled continuous operator dual of $\{\Lambda_{w}\}_{w\in\Omega}$ with corresponding invertible operator $K$ if and only if $C^{'}T_{\Lambda}^{\ast}T_{\Gamma}KC=I_{U}$ where  $T_{\Lambda}$ and $T_{\Gamma}$ are   $(C$-$C^{'})$-controlled continuous $\ast$-$g$-frame transforms of $\{\Lambda_{w}\}_{w\in\Omega}$ and $\{\Gamma_{w}\}_{w\in\Omega}$, respectively.
\end{remark}

\begin{theorem}
	Let $\{\Lambda_{w}\}_{w\in\Omega}$ be a $(C$-$C^{'})$-controlled continuous $\ast$-$g$-frame for $U$ with $(C$-$C^{'})$-controlled continuous $\ast$-$g$-frame transform and $Q\in End_{\mathcal{A}}^{\ast}(U)$ be invertible. The set of operator duals of $(C$-$C^{'})$-controlled continuous $\ast$-$g$-frame for $\{\Lambda_{w}\}_{w\in\Omega}$ is one-to-one correspondence  to the set of $(C$-$C^{'})$-controlled operator duals of $(C$-$C^{'})$-controlled continuous $\ast$-$g$-frame for $\{\Lambda_{w}Q^{\ast}\}_{w\in\Omega}$ with corresponding invertible operator.
\end{theorem}

\begin{proof}
	Let $T_{\Lambda Q^{\ast}}$ be a $(C$-$C^{'})$-controlled continuous $\ast$-$g$-frame transform of $\{\Lambda_{w}Q^{\ast}\}_{w\in\Omega}$. 
	
	Suppose that $\{\Gamma_{w}\}_{w\in\Omega}$ is  an $(C$-$C^{'})$-controlled operator dual for $\{\Lambda_{w}\}_{w\in\Omega}$. Then 
	\begin{align*}
		I_{U}&=C^{'}T^{\ast}_{\Lambda}T_{\Gamma}KC\\&=C^{'}QT_{\Lambda}^{\ast}T_{\Gamma}Q^{\ast}(Q^{\ast})^{-1}CKQ^{-1}\\&=C^{'}T_{\Lambda Q^{\ast}}^{\ast}T_{\Gamma Q^{\ast}}(Q^{\ast})^{-1}CKQ^{-1}.
	\end{align*}
	So $\{\Lambda_{w}Q^{\ast}\}_{w\in\Omega}$ is an operator $(C$-$C^{'})$-controlled dual of $\{\Lambda_{w}Q^{\ast}\}_{w\in\Omega}$ with corresponding invertible operator $(Q^{\ast})^{-1}KQ^{-1}$.\\
	Conversely, assume that $\{\Gamma_{w}\}_{w\in\Omega}$ is an $(C$-$C^{'})$-controlled operator dual for $\{\Lambda_{w}Q^{\ast}\}_{w\in\Omega}$ with corresponding invertible operator $K$ and $(C$-$C^{'})$-controlled continuous $\ast$-$g$-frame transform $T_{\Gamma}$. Then 
	\begin{align*}
		I_{U}=C^{'}T_{\Lambda Q^{\ast}}^{\ast}T_{\Gamma}CK&\implies I_{U}=QC^{'}T_{\Lambda}^{\ast}T_{\Gamma}CK\\&\implies Q^{-1}=C^{'}T_{\Lambda}^{\ast}T_{\Gamma}CK\\&\implies C^{'}T_{\Lambda}^{\ast}T_{\Gamma}Q^{\ast}(Q^{\ast})^{-1}CK=Q^{-1}\\&\implies C^{'}T_{\Lambda}^{\ast}T_{\Gamma Q^{\ast}}(Q^{\ast})^{-1}KCQ=I_{U}.
	\end{align*}
	Hence $\{\Gamma_{w}Q^{\ast}\}_{w\in\Omega}$ is an $(C$-$C^{'})$-controlled operator dual of $\{\Lambda_{w}\}_{w\in\Omega}$ with corresponding invertible operator $(Q^{\ast})^{-1}KQ$.
\end{proof}

\begin{theorem}
	Let $\{\Lambda_{w}\}_{w\in\Omega}$ be a $(C$-$C^{'})$-controlled continuous $\ast$-$g$-frame for $U$, $H$ be an orthogonally complemented submodule of $U$ and $P_{H}$ be the orthogonal projection on $H$.  Then the following statements hold
	\begin{itemize}
		\item[(1)] The set $\{\Lambda_{w}P_{H}\}_{w\in\Omega}$ is a $(C$-$C^{'})$-controlled continuous $\ast$-$g$-frame for $H$.
		\item[(2)] If $\{\Gamma_{w}\}_{w\in\Omega}$ is an $(C$-$C^{'})$-controlled operator dual of $\{\Lambda_{w}\}_{w\in\Omega}$  with corresponding invertible operator $K$ and $K(H)\subset H$, then $\{\Gamma_{w}P_{H}\}_{w\in\Omega}$ is an $(C$-$C^{'})$-controlled operator dual for $\{\Lambda_{w}P_{H}\}_{w\in\Omega}$ with corresponding invertible operator $K|_{H}$. 
		\item[(3)] If $S$ and $S_{P_{H}}$ are $(C$-$C^{'})$-controlled continuous $\ast$-$g$-frame operators of $\{\Lambda_{w}\}_{w\in\Omega}$ and $\{\Lambda_{w}P_{H}\}_{w\in\Omega}$ respectively, and for all $w\in\Omega$,  $S^{-1}_{P_{H}}P_{H}\Lambda_{w}^{\ast}=P_{H}S^{-1}\Lambda_{w}^{\ast}$, then $S_{P_{H}}^{-1}P_{H}=P_{H}S^{-1}$ on $H$.
	\end{itemize}
\end{theorem}

\begin{proof}
	$(1)$ We have for each $w\in\Omega$, $\Lambda_{w}P_{H}x=\Lambda_{w}x$, for all $x\in H$. Hence $\{\Lambda_{w}P_{H}\}_{w\in\Omega}$ is a $(C$-$C^{'})$-controlled continuous $\ast$-$g$-frame for $H$.
	
	$(2)$ Let $\{\Gamma_{w}\}_{w\in\Omega}$ be  an $(C$-$C^{'})$-controlled operator dual of $\{\Lambda_{w}\}_{w\in\Omega}$ such that $K(H)\subset H$ such that  $P_{H}Kx=Kx$ for each $x\in H$. Then 
	\begin{align*}		x&=P_{H}x\\&=P_{H}\left(\int_{\Omega}C^{'}\Lambda_{w}^{\ast}\Gamma_{w}KCxd\mu(w)\right)\\&=\int_{\Omega}C^{'}P_{H}\Lambda_{w}^{\ast}\Gamma_{w}KCxd\mu(w)\\&=\int_{\Omega}C^{'}(\Lambda_{w}P_{H})^{\ast}(\Gamma_{w}P_{H})KCxd\mu(w).
	\end{align*}
	Then $\{\Gamma_{w}P_{H}\}_{w\in\Omega}$ is an  $(C$-$C^{'})$-controlled operator dual for $\{\Lambda_{w}P_{H}\}_{w\in\Omega}$ with corresponding invertible operator $K_{/H}$.
	
	$(3)$ Suppose that $P_{H}S^{-1}\Lambda_{w}^{\ast}=S_{P_{H}}^{-1}P_{H}\Lambda_{w}^{\ast}$ for each $w\in\Omega$. Then  we have for each $x\in H$
	\begin{align*}
		S_{P_{H}}^{-1}P_{H}x&=S_{P_{H}}^{-1}P_{H}\left(\int_{\Omega}C^{'}\Lambda_{w}^{\ast}\Lambda_{w}S^{-1}Cxd\mu(w)\right)\\&=\int_{\Omega}S_{P_{H}}^{-1}P_{H}\Lambda_{w}^{\ast}\Lambda_{w}S^{-1}xd\mu(w)\\&=\int_{\Omega}C^{'}P_{H}S^{-1}\Lambda_{w}^{\ast}\Lambda_{w}S^{-1}Cxd\mu(w)\\&=P_{H}\int_{\Omega}C^{'}(\Lambda_{w}S^{-1})^{\ast}C(\Lambda_{w}S^{-1})xd\mu(w)\\&=P_{H}S^{-1}x.
	\end{align*}
	This completes the proof.
\end{proof}

\begin{proposition}\label{02}
	Let $\{\Gamma_{w}\}_{w\in\Omega}$ be a $(C$-$C^{'})$-controlled operator dual for $\{\Lambda_{w}\}_{w\in\Omega}$ with corresponding invertible operator $K$. Then $\{\Lambda_{w}\}_{w\in\Omega}$ is a $(C$-$C^{'})$-controlled operator dual for $\{\Gamma_{w}\}_{w\in\Omega}$ with corresponding invertible operator $K^{\ast}$.
\end{proposition}

\begin{proof}
	Since $\{\Lambda_{w}\}_{w\in\Omega}$ and $\{\Gamma_{w}\}_{w\in\Omega}$ are $(C$-$C^{'})$-controlled continuous $\ast$-$g$-frames,  we have the $(C$-$C^{'})$-controlled continuous $\ast$-$g$-frame transforms  $T_{\Lambda}$ and $T_{\Gamma}$ for $\{\Lambda_{w}\}_{w\in\Omega}$ and $\{\Gamma_{w}\}_{w\in\Omega}$, respectively.
	
	By definition of $(C$-$C^{'})$-controlled operator dual, 
	\begin{equation*}
		x=\int_{\Omega}C^{'}\Lambda_{w}^{\ast}\Gamma_{w}KCxd\mu(w),\qquad \forall x\in U.
	\end{equation*}
	Thus  $C^{'}T_{\Lambda}^{\ast}T_{\Gamma}CK=I_{U}$. Since  $K$ is invertible,  we have  $K^{-1}=C^{'}T_{\Lambda}^{\ast}T_{\Gamma}C$ and  so 
	\begin{equation*}
		I_{U}=KC^{'}(T_{\Lambda}^{\ast}T_{\Gamma})C \implies I_{U}=C^{'}T^{\ast}_{\Gamma}T_{\Lambda}CK^{\ast}.
	\end{equation*}
	Therefore, 
	\begin{equation*}
		x=\int_{\Omega}C^{'}\Gamma_{w}^{\ast}\Lambda_{w}K^{\ast}Cxd\mu(w),\qquad\forall x\in U.
	\end{equation*}
	We conclude that $\{\Lambda_{w}\}_{w\in\Omega}$ is an $(C$-$C^{'})$-controlled operator dual for $\{\Gamma_{w}\}_{w\in\Omega}$ with corresponding invertible operator $K^{\ast}$.
\end{proof}
\begin{theorem}\label{33}
	Let $\{\Lambda_{w}\}_{w\in \Omega}$ and $\{\Gamma_{w}\}_{w\in \Omega}$ be a $(C$-$C^{'})$-controlled  continuous $\ast$-$g$-Bessel sequences for $U$ with $(C$-$C^{'})$-controlled continuous $\ast$-$g$-frame transforms $T_{\Lambda}$ and $T_{\Gamma}$, respectively. If there exists an adjointable and invertible operator $K$ on $U$ such that
	\begin{equation*}
		x=\int_{\Omega}C^{'}\Lambda_{w}^{\ast}\Gamma_{w}KCx d\mu(w),\qquad\forall x\in U, 
	\end{equation*}
	then $\{\Gamma_{w}\}_{w\in \Omega}$ is the $(C$-$C^{'})$-controlled continuous operator duals of $\{\Lambda_{w}\}_{w\in\Omega}$ with corresponding invertible operator $K$ and $\{\Lambda_{w}\}_{w\in \Omega}$ is the $(C$-$C^{'})$-controlled continuous operator duals of $\{\Gamma_{w}\}_{w\in\Omega}$ with corresponding invertible operator $K^{\ast}$.
\end{theorem}

\begin{proof}
	We have for each $x\in U$
	\begin{align*}
		\langle x,x\rangle&=\langle T^{\ast}_{\Lambda}T_{\Gamma}Kx,  T^{\ast}_{\Lambda}T_{\Gamma}Kx\rangle\\&\leq ||T_{\Lambda}||^{2}\langle T_{\Gamma}Kx, T_{\Gamma}Kx\rangle\\&\leq ||T_{\Lambda}||^{2}\int_{\Omega}\langle \Gamma_{w}KCx,\Gamma_{w}KC^{'}x\rangle d\mu(w).
	\end{align*}
	Then
	\begin{equation*}
		||T_{\Lambda}||^{-2}\langle K^{-1}x,K^{-1}x\rangle \leq\int_{\Omega}\langle \Gamma_{w}KK^{-1}Cx, \Gamma_{w}KK^{-1}C^{'}x\rangle d\mu(w).
	\end{equation*}
	By Lemma \ref{3},  for each $x\in U$, we have
	\begin{equation*}
		||(K^{-1}(K^{-1})^{\ast})^{-1}||^{-1}\langle x,x\rangle \leq \langle K^{-1}Cx,K^{-1}C^{'}x\rangle.
	\end{equation*}
	Hence
	\begin{equation*}
		||T_{\Lambda}||^{-2}||(K^{-1}(K^{-1})^{\ast})^{-1}||^{-1}\langle x,x\rangle\leq\int_{\Omega}\langle \Gamma_{w}Cx,\Gamma_{w}C^{'}x\rangle d\mu(w).
	\end{equation*}
	We put $A=||(K^{-1}(K^{-1})^{\ast})^{-1}||^{-\frac{1}{2}}$. Then for each $x\in U$,
	\begin{equation*}
		||T_{\Lambda}||^{-1}A1_{\mathcal{A}}\langle x,x\rangle (	||T_{\Lambda}||^{-1}A1_{\mathcal{A}})^{\ast}\leq\int_{\Omega}\langle \Gamma_{w}Cx,\Gamma_{w}C^{'}x\rangle d\mu(w).
	\end{equation*}
	Therefore, $\{\Gamma_{w}\}_{w\in\Omega}$ is a $(C$-$C^{'})$-controlled continuous $\ast$-$g$-frame sequence for $U$.\\
	Similarly, $\{\Lambda_{w}\}_{w\in\Omega}$ is a $(C$-$C^{'})$-controlled continuous $\ast$-$g$-frame sequence for $U$. So by Proposition \ref{02}, $\{\Lambda_{w}\}_{w\in\Omega}$ and $\{\Gamma_{w}\}_{w\in\Omega}$ are the $(C$-$C^{'})$-controlled continuous operator duals to each other.
\end{proof}
\begin{theorem}\label{55}
	Let $\{\Lambda_{w}\}_{w\in\Omega}$ be a $(C$-$C)$-controlled continuous $\ast$-$g$-frame for $U$ with $(C$-$C)$-controlled continuous $\ast$-$g$-frame transform $T_{\Lambda}$ and $(C$-$C)$-controlled continuous $\ast$-$g$-frame operator $S$. If $K$ is an invertible adjointable operator on $U$, then  the set $\mathcal{C}$ of all right inverses of $KT_{\Lambda}^{\ast}$ is
	\begin{equation*}
		\biggl\{T_{\Lambda}S^{-1}K^{-1}+(I_{U}-T_{\Lambda}S^{-1}T_{\Lambda}^{\ast})\theta;\; \theta\in End_{\mathcal{A}}^{\ast}(U,\oplus_{w\in\Omega}V_{w})\biggr\}.
	\end{equation*}
\end{theorem}

\begin{proof}
	Let $G\in End_{\mathcal{A}}^{\ast}(U,\oplus_{w\in\Omega}V_{w})$ be a right inverse of $KT_{\Lambda}^{\ast}$.
	Then we have 
	\begin{align*}
		G&=T_{\Lambda}S^{-1}K^{-1}+G-T_{\Lambda}S^{-1}K^{-1}\\&=T_{\Lambda}S^{-1}K^{-1}+G-T_{\Lambda}S^{-1}K^{-1}KT_{\Lambda}^{\ast}G\\&=T_{\Lambda}S^{-1}K^{-1}+(I_{U}-T_{\Lambda}S^{-1}T_{\Lambda}^{\ast})G.
	\end{align*}
	Then it is enough to set $\theta=G$.\\
	Conversely, let  $\theta\in End_{\mathcal{A}}^{\ast}(U,\oplus_{w\in\Omega}V_{w})$.
	Then we have 
	\begin{align*}
		KT_{\Lambda}^{\ast}(T_{\Lambda}S^{-1}K^{-1}+(I_{U}-T_{\Lambda}S^{-1}T^{\ast}_{\Lambda})\theta)&=KT^{\ast}_{\Lambda}T_{\Lambda}S^{-1}K^{-1}+KT^{\ast}_{\Lambda}\theta-KT^{\ast}_{\Lambda}T_{\Lambda}S^{-1}T^{\ast}_{\Lambda}\theta\\&=I_{U}+KT_{\Lambda}^{\ast}\theta-KT^{\ast}_{\Lambda}\theta\\&=I_{U}.
	\end{align*}
	Therefore, $T_{\Lambda}S^{-1}K^{-1}+(I_{U}-T_{\Lambda}S^{-1}T^{\ast}_{\Lambda})\theta$ is a right inverse of $KT^{\ast}_{\Lambda}$.
\end{proof}

\begin{theorem}
	Let $\{\Gamma_{w}\}_{w\in \Omega}$ be a $(C$-$C^{'})$-controlled operator dual of the $(C$-$C^{'})$-controlled continuous $\ast$-$g$-frame of $\{\Lambda_{w}\}_{w\in \Omega}$ with corresponding invertible operator $K$, and let $\{\Lambda_{w}S^{-1}\}_{w\in \Omega}$ be the $(C$-$C^{'})$-controlled canonical dual of $\{\Lambda_{w}\}_{w\in\Omega}$. If $v$ is a strictly nonzero element in the center of $\mathcal{A}$ and $\theta_{w}=v\Gamma_{w}+v\Lambda_{w}S^{-1}K^{-1}$ for $w\in \Omega$, then $\{\theta_{w}\}_{w\in\Omega}$ is a $(C$-$C^{'})$-controlled operator dual of $\{\Lambda_{w}\}_{w\in\Omega}$ with corresponding invertible operator $\frac{1}{2}v^{-1}K$.
\end{theorem}

\begin{proof}
	Suppose that $\{\Gamma_{w}\}_{w\in \Omega}$ is  a $(C$-$C^{'})$-controlled operator dual of the $(C$-$C^{'})$-controlled continuous $\ast$-$g$-frame of $\{\Lambda_{w}\}_{w\in \Omega}$.
	Then we have for each $x\in U$,
	\begin{align*}
		\int_{\Omega}C^{'}\Lambda_{w}^{\ast}\theta_{w}\left(\frac{1}{2}v^{-1}K\right)Cxd\mu(w)&=\int_{\Omega}C^{'}\Lambda_{w}^{\ast}\bigg(v\Gamma_{w}+v\Lambda_{w}S^{-1}K^{-1}\bigg)\frac{1}{2}v^{-1}KCxd\mu(w)\\&=\int_{\Omega}C^{'}\Lambda_{w}^{\ast}\left(\frac{1}{2}\Gamma_{w}K+\frac{1}{2}\Lambda_{w}S^{-1}\right)Cxd\mu(w)\\&=\frac{1}{2}\int_{\Omega}C^{'}\Lambda_{w}^{\ast}\Gamma_{w}KCxd\mu(w)+\frac{1}{2}\int_{\Omega}C^{'}\Lambda_{w}^{\ast}\Lambda_{w}S^{-1}Cxd\mu(w)\\&=\frac{1}{2}x+\frac{1}{2}x=x.
	\end{align*}
	By Theorem \ref{33},  $\{\theta_{w}\}_{w\in\Omega}$ is a $(C$-$C^{'})$-controlled operator dual of $\{\Lambda_{w}\}_{w\in\Omega}$.
\end{proof}

\begin{remark}
	By  Proposition \ref{02} and Theorem \ref{33},  $\{\Gamma_{w}\}_{w\in \Omega}$ is a $(C$-$C^{'})$-controlled operator dual of $\{\Lambda_{w}\}_{w\in\Omega}$ if and only if $T_{\Gamma}$ is a right inverse of $KT^{\ast}_{\Lambda}$, where   $T_{\Lambda}$ and $T_{\Gamma}$ are $(C$-$C^{'})$-controlled continuous $\ast$-$g$-frame transforms of $\{\Lambda_{w}\}_{w\in\Omega}$ and $\{\Gamma_{w}\}_{w\in \Omega}$, respectively. So we can characterize all of the $(C$-$C^{'})$-controlled operator duals of  $\{\Lambda_{w}\}_{w\in\Omega}$ by  a set of all right inverses of $KT_{\Lambda}^{\ast}$.
\end{remark}

\begin{theorem}\label{12}
	The set of all $(C$-$C^{'})$-controlled continuous $\ast$-$g$-Bessel family for $U$ with respect to $\{V_{w}: w\in\Omega\}$ is
	\begin{equation*}
		\biggl\{\{P_{w}\theta\}_{w\in\Omega}: \theta\in  End_{\mathcal{A}}^{\ast}(U,\oplus_{w\in\Omega}V_{w}) \biggr\},
	\end{equation*}
	where $P_{w}$ is the orthogonal projection on $V_{w}$.
\end{theorem}

\begin{proof}
	Let $\{\Lambda_{w}\}_{w\in\Omega}$ be a $(C$-$C^{'})$-controlled continuous $\ast$-$g$-Bessel sequence for $U$ with bound $B$. Then  we have for each $x\in U$, $P_{w}(\{\Lambda_{w}x\}_{w\in\Omega})=\Lambda_{w}x$ and  hence $P_{w}Tx=\Lambda_{w}x$ with $T$ the $(C$-$C^{'})$-controlled continuous $\ast$-$g$-frame transform of $\{\Lambda_{w}\}_{w\in\Omega}$ and so $P_{w}T=\Lambda_{w}$. Thus  for each $x\in U$,
	\begin{equation*}
		\int_{\Omega}\langle \Lambda_{w}C^{'}x,\Lambda_{w}Cx\rangle d\mu(w)=\int_{\Omega}\langle P_{w}TC^{'}x,P_{w}TCx\rangle d\mu(w)\leq B\langle x,x\rangle B^{\ast}.	
	\end{equation*}
	We conclude that $\{P_{w}T\}_{w\in\Omega}$ is a $(C$-$C^{'})$-controlled continuous  $\ast$-$g$-Bessel sequence for $U$.\\
	For the converse let $\theta\in  End_{\mathcal{A}}^{\ast}(U,\oplus_{w\in\Omega}V_{w})$ then,
	\begin{align*}
		\int_{\Omega}\langle P_{w}\theta C^{'} x, P_{w}\theta C x\rangle d\mu(w)&=\int_{\Omega}\langle \theta_{w}C^{'}x,\theta_{w}Cx\rangle d\mu(w)\\
		&=\langle \theta C^{'}x,\theta Cx\rangle\leq\|\theta\|^{2}\langle C^{'}x, Cx\rangle\\
		&= (\|\theta\|1_{\mathcal{A}})\langle x,x\rangle (\|\theta\|1_{\mathcal{A}})^{\ast}.	
	\end{align*}
	So $\{P_{w}\theta\}_{w\in\Omega}$ is a $(C$-$C^{'})$-controlled continuous $\ast$-$g$-Bessel family for $U$ with upper bound $\|\theta\|1_{\mathcal{A}}$.
\end{proof}

\begin{theorem}\label{66}
	Let $\{\Lambda_{w}\}_{w\in\Omega}$ be a $(C$-$C^{'})$-controlled continuous $\ast$-$g$-frame for $U$ with $(C$-$C^{'})$-controlled continuous $\ast$-$g$-frame transform $T$. If $\theta : U\rightarrow \oplus_{w\in\Omega}V_{w}$ is an adjointable right inverse operator of $KT^{\ast}$, then $\{P_{w}\theta\}_{w\in\Omega}$ is a $(C$-$C^{'})$-controlled operator dual of $\{\Lambda_{w}\}_{w\in\Omega}$ with the corresponding invertible operator $K$.
\end{theorem}

\begin{proof}
	By Theorem \ref{12}, $\{P_{w}\theta\}_{w\in\Omega}$ is a $(C$-$C^{'})$-controlled continuous $\ast$-$g$-Bessel sequence for $U$. Thus  $KT^{\ast}\theta=I_{U}$. Since  $\theta$ is a right inverse of $KT^{\ast}$,  $I_{U}=\theta^{\ast}TK^{\ast}$ and  hence $\theta^{\ast}$ is surjective. So by Lemma \ref{3},  
	\begin{equation*}
		\|(\theta^{\ast}\theta)^{-1}\|^{-1}\langle x,x\rangle \leq \langle \theta x,\theta x\rangle=\int_{\Omega}\langle P_{w}\theta C^{'} x,P_{w}\theta C x\rangle d\mu(w)\leq\|\theta\|^{2}\langle x,x\rangle, \;\forall x\in U.
	\end{equation*}
	Thus $\{P_{w}\theta\}_{w\in\Omega}$ is a $(C$-$C^{'})$-controlled continuous $\ast$-$g$-frame sequence for $U$ with $(C$-$C^{'})$-controlled continuous $\ast$-$g$-frame transform $\theta$. So  $x=C^{'}\theta^{\ast}TK^{\ast}Cx=\int_{\Omega}C^{'}(P_{w}\theta)^{\ast}\Lambda_{w} K^{\ast}Cxd\mu(w)$ for each $x\in U$.
	Thus we obtain that  $\{P_{w}\theta\}_{w\in\Omega}$ is a $(C$-$C^{'})$-controlled operator dual of $\{\Lambda_{w}\}_{w\in\Omega}$ with corresponding invertible operator $K^{\ast}$, the proof is complete by Theorem \ref{33}.
\end{proof}
\begin{theorem}
	Let $\{\Lambda_{w}\}_{w\in\Omega}$ be a $(C$-$C)$-controlled continuous $\ast$-$g$-frame for $U$ with $(C$-$C)$-controlled continuous $\ast$-$g$-frame transform  $T$ and $(C$-$C)$-controlled continuous $\ast$-$g$-frame operator $S$. Let $K$ be an invertible operator on $U$ and $\{G_{w}\}_{w\in\Omega}$ be a $(C$-$C)$-controlled continuous $\ast$-$g$-Bessel sequence for $U$. Then every $(C$-$C)$-controlled operator dual for $\{\Lambda_{w}\}_{w\in\Omega}$ is of the form 
	\begin{equation*}		\Lambda_{w}CS^{-1}K^{-1}+G_{w}C-\int_{\Omega}\Lambda_{w}CS^{-1}\Lambda^{\ast}_{t}G_{t}.
	\end{equation*}
\end{theorem}

\begin{proof}
	Suppose that $\{\Lambda_{w}\}$ is a $(C$-$C)$-controlled continuous $\ast$-$g$-frame for $U$ with upper $(C$-$C)$-controlled continuous $\ast$-$g$-frame bound $B$. Let $T_{G}$ be a $(C$-$C)$-controlled continuous $\ast$-$g$-frame transform of $(C$-$C)$-controlled continuous $\ast$-$g$-Bessel sequence $\{G_{w}\}_{w\in\Omega}$ with $(C$-$C)$-controlled continuous $\ast$-$g$-frame bound $E$. \\
	Put for every $w\in\Omega$,
	\begin{equation*}	\theta_{w}C=\Lambda_{w}CS^{-1}K^{-1}+G_{w}C-\int_{\Omega}\Lambda_{w}CS^{-1}\Lambda^{\ast}_{t}G_{t}.
	\end{equation*}
	Then, for each $x\in U$, we have 
	\begin{align}\label{eq15}
		\theta_{w}C x&=\Lambda_{w}CS^{-1}K^{-1}x+G_{w}Cx-\int_{\Omega}\Lambda_{w}S^{-1}\Lambda_{t}^{\ast}G_{t}Cxd\mu(t)\notag\\&=\Lambda_{w}CS^{-1}K^{-1}x+G_{w}Cx-\Lambda_{w}S^{-1}\int_{\Omega}C\Lambda_{t}^{\ast}G_{t}xd\mu(t)\notag\\&=\Lambda_{w}CS^{-1}K^{-1}x+G_{w}Cx-\Lambda_{w}CS^{-1}T^{\ast}T_{G}x.
	\end{align}
	By  \eqref{eq15}, for each $x\in U$, we have 
	\begin{align*}
		\|\{\theta_{w}Cx\}_{w\in\Omega}\|&=\left\|\biggl\{\Lambda_{w}CS^{-1}K^{-1}x+G_{w}Cx-\Lambda_{w}CS^{-1}T^{\ast}T_{G}x\biggr\}_{w\in\Omega}\right\|\\&\leq \|\{\Lambda_{w}CS^{-1}K^{-1}x\}_{w\in\Omega}\|+\|\{G_{w}Cx\}_{w\in\Omega}\|+\|\{\Lambda_{w}CS^{-1}T^{\ast}T_{G}x\}_{w\in\Omega}\|\\&\leq \left\|\int_{\Omega}\langle \Lambda_{w}CS^{-1}K^{-1}x,\Lambda_{w}CS^{-1}K^{-1}x\rangle d\mu(w)\right\|^{\frac{1}{2}}+\left\|\int_{\Omega}\langle G_{w}Cx,G_{w}Cx\rangle d\mu(w)\right\|^{\frac{1}{2}}\\&\qquad\qquad\qquad+\left\|\int_{\Omega}\langle \Lambda_{w}CS^{-1}T^{\ast}T_{G}x,\Lambda_{w}CS^{-1}T^{\ast}T_{G}x\rangle d\mu(w)\right\|^{\frac{1}{2}}\\&\leq ||B\langle S^{-1}K^{-1}x, S^{-1}K^{-1}x\rangle B^{\ast}||^{\frac{1}{2}}+||E\langle x,x\rangle E^{\ast}||^{\frac{1}{2}}\\&\qquad\qquad+||B\langle S^{-1}T^{\ast}T_{G}x,S^{-1}T^{\ast}T_{G}x\rangle B^{\ast}||^{\frac{1}{2}}\\&\leq \big(\|B\| \|S^{-1}\| \|K^{-1}\|+\|E\|+\|S^{-1}\| \|T^{\ast}\| \|T_{G}\|\big)\|x\|.
	\end{align*} 
	Then we can define the operator $\phi: U\rightarrow \oplus_{w\in\Omega}V_{w}$ by $\phi(x)=\{\theta_{w}x\}_{w\in\Omega}$, clearly it is adjointable and we have for each $x\in U$
	\begin{align*}
		P_{w}\phi x&=P_{w}(\{\theta_{w}Cx\}_{w\in\Omega})\\&=\theta_{w}Cx\\&=\Lambda_{w}CS^{-1}K^{-1}x+G_{w}Cx-\Lambda_{w}CS^{-1}T^{\ast}T_{G}x\\&=P_{w}\left(\biggl\{\Lambda_{w}CS^{-1}K^{-1}x+G_{w}Cx-\Lambda_{w}CS^{-1}T^{\ast}T_{G}x\biggr\}_{w\in\Omega}\right)\\&=P_{w}\big(CTS^{-1}K^{-1}x+T_{G}Cx-TS^{-1}T^{\ast}T_{G}Cx\big).
	\end{align*}
	Hence 
	\begin{equation*}
		\phi=TS^{-1}K^{-1}+T_{G}-TS^{-1}T^{\ast}T_{G}=TS^{-1}K^{-1}+(I_{U}-TS^{-1}T^{\ast})T_{G}.
	\end{equation*}
	By Theorem \ref{55}, $\phi$ is a right inverse of $KT^{\ast}$, and by Theorem \ref{66}, $\{\theta_{w}\}_{w\in\Omega}$ is an operator dual of $\{\Lambda_{w}\}_{w\in\Omega}$.
\end{proof}
\section{Stability Problem for controlled continuous $\ast$-$g$-frame in Hilbert $C^{\ast}$-modules}\label{section5}

The question of stability plays an important role in various fields of applied mathematics. The classical theorem of the stability of a base is due to Paley and Wiener \cite{Paley}. It is based on the fact that a bounded operator T on a Banach space is invertible if  $\|I-T\|<1$.
\begin{theorem} \cite{Paley} 
	Let $\{f_{i}\}_{i\in\mathbb{N}}$ be a basis of a Banach space $X$, and $\{g_{i}\}_{i\in\mathbb{N}}$ be  a sequence of vectors in $X$. If there exists a constant $\lambda\in[0,1)$ such that
	\begin{equation*}
		\Big\|\sum_{i\in\mathbb{N}}c_{i}(f_{i}-g_{i})\Big\|\leq\lambda\Big\|\sum_{i\in\mathbb{N}}c_{i}f_{i}\Big\|
	\end{equation*}
	for all finite sequences  $\{c_{i}\}_{i\in\mathbb{N}}$ of scalars, then $\{g_{i}\}_{i\in\mathbb{N}}$ is also a basis for $X$.
\end{theorem}

\begin{theorem}
	Let $\{\Lambda_{w}\in End_{\mathcal{A}}^{\ast}(U,V_{w}): w\in\Omega\}$ be a $(C$-$C)$-controlled continuous $\ast$-$g$-frame for $U$, with lower and upper bounds $A$ and $B$, respectively. Let $\Gamma_{w}\in End_{\mathcal{A}}^{\ast}(U,V_{w})$ for any $w\in\Omega$. Then the following are equivalent
	\begin{itemize}
		\item [(1)] $\{\Gamma_{w}\in End_{\mathcal{A}}^{\ast}(U,V_{w}): w\in\Omega\}$ is a $(C$-$C)$-controlled continuous $\ast$-$g$-frame for $U$.
		\item[(2)] There exists a constant $M>0$ such that for any $x\in U$, one has
		\begin{multline}\label{5.1}
			\bigg\|\int_{\Omega}\langle(\Lambda_{w}-\Gamma_{w})Cx,(\Lambda_{w}-\Gamma_{w})Cx\rangle d\mu(w)\bigg\|\\
			\leq M\min\bigg\{\bigg\|\int_{\Omega}\langle\Lambda_{w}Cx,\Lambda_{w}Cx\rangle d\mu(w)\bigg\|,\bigg\|\int_{\Omega}\langle\Gamma_{w}Cx,\Gamma_{w}Cx\rangle d\mu(w)\bigg\|\bigg\}.
		\end{multline}
	\end{itemize}
\end{theorem}

\begin{proof}
	$(1)\Rightarrow(2)$. Suppose that $\{\Gamma_{w}\in End_{\mathcal{A}}^{\ast}(U,V_{w}): w\in\Omega\}$ is a $(C$-$C)$-controlled continuous $\ast$-$g$-frame for $U$ with lower and upper bounds $E$ and $F$, respectively. Then for any $x\in U$, we have
	\begin{align*}
		\bigg\|\int_{\Omega}\langle(\Lambda_{w}-\Gamma_{w})Cx,&(\Lambda_{w}-\Gamma_{w})Cx\rangle d\mu(w)\bigg\|^{\frac{1}{2}}=\big\|\{(\Lambda_{w}-\Gamma_{w})Cx\}_{w\in\Omega}\big\|\\
		&\leq\big\|\{\Lambda_{w}Cx\}_{x\in\Omega}\big\|+\big\|\{\Gamma_{w}Cx\}_{x\in\Omega}\big\|\\
		&=\bigg\|\int_{\Omega}\langle\Lambda_{w}Cx,\Lambda_{w}Cx\rangle d\mu(w)\bigg\|^{\frac{1}{2}}+\bigg\|\int_{\Omega}\langle\Gamma_{w}Cx,\Gamma_{w}Cx\rangle d\mu(w)\bigg\|^{\frac{1}{2}}\\
		&\leq\|B\|\|\langle x,x\rangle\|^{\frac{1}{2}}+\bigg\|\int_{\Omega}\langle\Gamma_{w}Cx,\Gamma_{w}Cx\rangle d\mu(w)\bigg\|^{\frac{1}{2}}\\
		&\leq\|B\|\|E^{-1}\|\bigg\|\int_{\Omega}\langle\Gamma_{w}Cx,\Gamma_{w}Cx\rangle d\mu(w)\bigg\|^{\frac{1}{2}}+\bigg\|\int_{\Omega}\langle\Gamma_{w}Cx,\Gamma_{w}Cx\rangle d\mu(w)\bigg\|^{\frac{1}{2}}\\
		&=\bigg(\|B\|\|E^{-1}\|+1\bigg)\bigg\|\int_{\Omega}\langle\Gamma_{w}Cx,\Gamma_{w}Cx\rangle d\mu(w)\bigg\|^{\frac{1}{2}}.
	\end{align*}
	Similarly,  we have
	\begin{align*}
		\bigg\|\int_{\Omega}\langle(\Lambda_{w}-\Gamma_{w})Cx,(\Lambda_{w}-\Gamma_{w})Cx\rangle d\mu(w)\bigg\|^{\frac{1}{2}}\leq\bigg(\|F\|\|A^{-1}\|+1\bigg)\bigg\|\int_{\Omega}\langle\Lambda_{w}Cx,\Lambda_{w}Cx\rangle d\mu(w)\bigg\|^{\frac{1}{2}}.
	\end{align*}
	Let $M=\min\Bigg\{\bigg(\|B\|\|E^{-1}\|+1\bigg)^{2},\bigg(\|F\|\|A^{-1}\|+1\bigg)^{2}\Bigg\}$.  Then \eqref{5.1} holds.
	
	$(2)\Rightarrow(1)$. Suppose that  \eqref{5.1} holds. For any $x\in U$, we have
	\begin{align*}
		\|A^{-1}\|^{-1}\|\langle x,x\rangle\|^{\frac{1}{2}}&\leq\bigg\|\int_{\Omega}\langle\Lambda_{w}Cx,\Lambda_{w}Cx\rangle d\mu(w)\bigg\|^{\frac{1}{2}}\\
		&\leq\bigg\|\int_{\Omega}\langle(\Lambda_{w}-\Gamma_{w})Cx,(\Lambda_{w}-\Gamma_{w})Cx\rangle d\mu(w)\bigg\|^{\frac{1}{2}}+\bigg\|\int_{\Omega}\langle\Gamma_{w}Cx,\Gamma_{w}Cx\rangle d\mu(w)\bigg\|^{\frac{1}{2}}\\
		&\leq M^{\frac{1}{2}}\bigg\|\int_{\Omega}\langle\Gamma_{w}Cx,\Gamma_{w}Cx\rangle d\mu(w)\bigg\|^{\frac{1}{2}}+\bigg\|\int_{\Omega}\langle\Gamma_{w}Cx,\Gamma_{w}Cx\rangle d\mu(w)\bigg\|^{\frac{1}{2}}\\
		&=\big(1+M^{\frac{1}{2}}\big)\bigg\|\int_{\Omega}\langle\Gamma_{w}Cx,\Gamma_{w}Cx\rangle d\mu(w)\bigg\|^{\frac{1}{2}}.
	\end{align*}
	Also we obtain
	\begin{align*}
		\bigg\|\int_{\Omega}\langle\Gamma_{w}Cx,\Gamma_{w}Cx\rangle d\mu(w)\bigg\|^{\frac{1}{2}}
		&\leq\bigg\|\int_{\Omega}\langle(\Lambda_{w}-\Gamma_{w})Cx,(\Lambda_{w}-\Gamma_{w})Cx\rangle d\mu(w)\bigg\|^{\frac{1}{2}}+\bigg\|\int_{\Omega}\langle\Lambda_{w}Cx,\Lambda_{w}Cx\rangle d\mu(w)\bigg\|^{\frac{1}{2}}\\
		&\leq M^{\frac{1}{2}}\bigg\|\int_{\Omega}\langle\Lambda_{w}Cx,\Lambda_{w}Cx\rangle d\mu(w)\bigg\|^{\frac{1}{2}}+\bigg\|\int_{\Omega}\langle\Lambda_{w}Cx,\Lambda_{w}Cx\rangle d\mu(w)\bigg\|^{\frac{1}{2}}\\
		&=\big(1+M^{\frac{1}{2}}\big)\bigg\|\int_{\Omega}\langle\Lambda_{w}Cx,\Lambda_{w}Cx\rangle d\mu(w)\bigg\|^{\frac{1}{2}}
		&\leq\big(1+M^{\frac{1}{2}}\big)\|B\|\|\langle x,x\rangle\|^{\frac{1}{2}}.
	\end{align*}
	Thus  $\{\Gamma_{w}\in End_{\mathcal{A}}^{\ast}(U,V_{w}): w\in\Omega\}$ is a $(C$-$C)$-controlled continuous $\ast$-$g$-frame for $U$.
\end{proof}

\begin{theorem}
	Let $\{\Lambda_{w}\}_{w\in\Omega}$ be a $(C$-$C)$-controlled continuous $\ast$-$g$-frame for $U$ with bounds $A$ and $B$. If $\{\Gamma_{w}\}_{w\in\Omega}$ is a $(C$-$C)$-controlled continuous $\ast$-$g$-Bessel sequence with bound $E$ such that $||A^{-1}||^{-1}\geq ||E||$, then $\{\Gamma_{w}+\Lambda_{w}\}_{w\in\Omega}$ is a $(C$-$C)$-controlled continuous $\ast$-$g$-frame for $U$.
\end{theorem}

\begin{proof}
	Let $x\in U$, we have 
	\begin{align*}
		\left\|\int_{\Omega}\langle (\Lambda_{w}+\Gamma_{w})Cx,(\Lambda_{w}+\Gamma_{w})Cx\rangle d\mu(w)\right\|^{\frac{1}{2}}&=||\{(\Lambda_{w}+\Gamma_{w})C\}_{w\in\Omega}||\\&\leq ||\{\Lambda_{w}Cx\}_{w\in\Omega}||+||\{\Gamma_{w}Cx\}_{w\in\Omega}||\\&\leq \left\|\int_{\Omega}\langle \Lambda_{w}Cx,\Lambda_{w}Cx\rangle d\mu(w)\right\|^{\frac{1}{2}}+	\left\|\int_{\Omega}\langle \Gamma_{w}Cx,\Gamma_{w}Cx\rangle d\mu(w)\right\|^{\frac{1}{2}} \\&\leq ||B\langle x,x\rangle B^{\ast} ||^{\frac{1}{2}}+||E\langle x,x\rangle E^{\ast} ||^{\frac{1}{2}}\\&\leq ||B||||x||+||E||||x||\\&\leq \big(||B||+||E||\big)||x||.
	\end{align*}
	Thus 
	\begin{equation}\label{eq5.2}
		\left\|\int_{\Omega}\langle (\Lambda_{w}+\Gamma_{w})Cx,(\Lambda_{w}+\Gamma_{w})Cx\rangle d\mu(w)\right\|\leq\big(||B||+||E||\big)^{2}||x||^{2}.
	\end{equation}
	On the other hand, 
	\begin{align*}
		\left\|\int_{\Omega}\langle (\Lambda_{w}+\Gamma_{w})Cx,(\Lambda_{w}+\Gamma_{w})Cx\rangle d\mu(w)\right\|^{\frac{1}{2}}&=||\{(\Lambda_{w}+\Gamma_{w})C\}_{w\in\Omega}||\\&\geq ||\{\Lambda_{w}C\}_{w\in\Omega}||-||\{\Gamma_{w}C\}_{w\in\Omega}||\\&\geq \left\|\int_{\Omega}\langle \Lambda_{w}Cx,\Lambda_{w}Cx\rangle d\mu(w)\right\|^{\frac{1}{2}}-\left\|\int_{\Omega}\langle \Gamma_{w}Cx,\Gamma_{w}Cx\rangle d\mu(w)\right\|^{\frac{1}{2}}\\&\geq ||A^{-1}||^{-1}||x||-||E||||x||\\&\geq (||A^{-1}||^{-1}-||E||)||x||. 
	\end{align*}
	Hence 
	\begin{equation}\label{eq5.3}
		(||A^{-1}||^{-1}-||E||)||x||\leq \left\|\int_{\Omega}\langle (\Lambda_{w}+\Gamma_{w})Cx,(\Lambda_{w}+\Gamma_{w})Cx\rangle d\mu(w)\right\|^{\frac{1}{2}}. 
	\end{equation}
	Therefore, from \eqref{eq5.2} and 
	\eqref{eq5.3}, $\{(\Lambda_{w}+\Gamma_{w})\}_{w\in\Omega}$ is a  $(C$-$C)$-controlled continuous $\ast$-$g$-frame for $U$.
\end{proof}
\begin{theorem}
	Let $\{T_{w}\}_{w \in \Omega}$ be a$(C$-$C)$-controlled continuous $\ast$-$g$-frame  for $End_{\mathcal{A}}^{\ast}(\mathcal{H})$ with bounds A and B, let $\{R_{w}\}_{w \in \Omega} \subset End_{\mathcal{A}}^{\ast}(\mathcal{H})$ and $\{\alpha_{w}\}_{w \in \Omega},\{\beta_{w}\}_{w \in \Omega} \in \mathbb{R}$ be two positively  family.
	If there exist two constants $0\leq \lambda, \mu<1$ such that for any $x \in\mathcal{H} $ we have 
	\begin{align*}
		\bigg\|\int_{\Omega}&\langle (\alpha_{w}T_{w}-\beta_{w}R_{w}) Cx,(\alpha_{w}T_{w}-\beta_{w}R_{w})C x\rangle_{\mathcal{A}} d\mu({\omega})\bigg\|^{\frac{1}{2}}\leq \\
		& \lambda \bigg\|\int_{\Omega}\langle \alpha_{w}T_{w}C x,\alpha_{w}T_{w}Cx\rangle_{\mathcal{A}} d\mu({\omega})\bigg\|^{\frac{1}{2}}
		+\mu\bigg\|\int_{\Omega}\langle \beta_{w}R_{w} Cx,\beta_{w}R_{w}C x\rangle_{\mathcal{A}} d\mu({\omega})\bigg\|^{\frac{1}{2}}.
	\end{align*}
	Then  $\{R_{w}\}_{w \in \Omega}$ is a $(C$-$C)$-controlled continuous $\ast$-$g$-frame for $End_{\mathcal{A}}^{\ast}(\mathcal{H})$.
	
\end{theorem}
\begin{proof}
	For every $x \in \mathcal{H}$, we have
	\begin{align*}
		\|\{\beta_{w}R_{w}Cx\}_{w \in \Omega}\|&\leq \|\{(\alpha_{w}T_{w}-\beta_{w}R_{w})Cx\}_{w \in \Omega}\| +\|\{\alpha_{w}T_{w}Cx\}_{w\in \Omega}\|\\
		&\leq \mu \|\{\beta_{w}R_{w}Cx\}_{w\in \Omega}\| + \lambda \|\{\alpha_{w}T_{w}Cx\}_{w\in \Omega}\|+ \|\{\alpha_{w}T_{w}Cx\}_{w\in \Omega}\|\\
		&=(1+\lambda) \|\{\alpha_{w}T_{w}Cx\}_{w\in \Omega}\|+ \mu \|\{\beta_{w}R_{w}Cx\}_{w \in \Omega}\|.
	\end{align*}
	Then,
	$$(1-\mu)\|\{\beta_{w}R_{w}Cx\}_{w\in \Omega}\|\leq(1+\lambda) \|\alpha_{w}T_{w}Cx\|. $$
	Therefore
	$$(1-\mu)\inf_{\omega \in \Omega} (\beta_{w})\|\{R_{w}Cx\}_{w \in \Omega}\|\leq(1+\lambda) \sup_{\omega \in \Omega}(\alpha_{w})\|\{T_{w}Cx\}_{w\in \Omega}\|. $$
	Hence 
	$$\|\{R_{w}Cx\}_{w\in \Omega}\| \leq \frac{(1+\lambda) \sup_{\omega \in \Omega}(\alpha_{w})}{(1-\mu)\inf_{\omega \in \Omega} (\beta_{w})}\|\{T_{w}Cx\}_{w\in \Omega}\|.$$
	Also, for all $x \in \mathcal{H}$, we have 
	\begin{align*}
		\|\{(\alpha_{w}T_{w}Cx\}_{w\in \Omega}\| &\leq \|\{(\alpha_{w}T_{w}-\beta_{w}R_{w})Cx\}_{w \in \Omega}\| +\|\{\beta_{w}R_{w}Cx\}_{w \in \Omega}\|\\
		&\leq \mu \|\{\beta_{w}R_{w}Cx\}_{w \in \Omega}\| + \lambda \|\{\alpha_{w}T_{w}Cx\}_{w\in \Omega}\|+ \|\{\alpha_{w}T_{w}Cx\}_{w\in \Omega}\|\\
		&= \lambda \|\{\alpha_{w}T_{w}Cx\}_{w\in \Omega}\|+(1+\mu)\|\{\beta_{w}R_{w}Cx\}_{w\in \Omega}\|,
	\end{align*}
	then
	$$(1- \lambda)\|\{\alpha_{w}T_{w}Cx\}_{w\in \Omega}\| \leq (1+\mu)\|\{\beta_{w}R_{w}Cx\}_{w\in \Omega}\|.$$
	Hence 
	$$(1- \lambda)\inf_{\omega \in \Omega} (\alpha_{w})\|\{T_{w}Cx\}_{w\in \Omega}\| \leq (1+\mu)\sup_{\omega \in \Omega}(\beta_{w})\|\{R_{w}Cx\}_{w\in \Omega}\|.$$
	Thus
	$$\frac{(1- \lambda)\inf_{\omega \in \Omega} (\alpha_{w})}{(1+\mu)\sup_{\omega \in \Omega}(\beta_{w})}\|\{T_{w}Cx\}_{w\in \Omega}\| \leq \|\{R_{w}Cx\}_{w\in \Omega}\|.$$
	Therefore
	$$A(\frac{(1- \lambda)\inf_{\omega \in \Omega} (\alpha_{w})}{(1+\mu)\sup_{\omega \in \Omega}(\beta_{w})}) \|\langle x,x\rangle_{\mathcal{A}}\|(\frac{(1- \lambda)\inf_{\omega \in \Omega} (\alpha_{w})}{(1+\mu)\sup_{\omega \in \Omega}(\beta_{w})})A^{*}\leq \|\{R_{w}Cx\}_{w}\|^2 .$$
	So,  
	\begin{align*}
		\|\{R_{w}Cx\}_{w\in \Omega}\|^2&\leq (\frac{(1+ \lambda)\sup_{\omega \in \Omega}(\alpha_{w})}{(1-\mu)\inf_{\omega \in \Omega} (\beta_{w})})^2  \|\{T_{w}Cx\}_{w\in \Omega}\|^2 \\\leq
		& B (\frac{(1+ \lambda)sup(\alpha_{w})}{(1-\mu)\inf_{\omega \in \Omega} (\beta_{w})})  \|\langle x,x \rangle_{\mathcal{A}}\|(\frac{(1+ \lambda)sup(\alpha_{w})}{(1-\mu)\inf_{\omega \in \Omega} (\beta_{w})})B^{*}  .
	\end{align*}
	Hence 
	\begin{align*}
		&A(\frac{(1- \lambda) \inf_{\omega \in \Omega} (\alpha_{w})}{(1+\mu) \sup_{\omega \in \Omega}(\beta_{w})}) \|\langle x,x\rangle_{\mathcal{A}}\|(\frac{(1- \lambda) \inf_{\omega \in \Omega} (\alpha_{w})}{(1+\mu) \sup_{\omega \in \Omega}(\beta_{w})})A^{*}\\&\leq \|\int_{\Omega}\langle R_{w}C x,R_{w}Cx \rangle_{\mathcal{A}} d\mu({\omega})\|\\
		&\leq  B (\frac{(1+ \lambda) \sup_{\omega \in \Omega}(\alpha_{w})}{(1-\mu) \inf_{\omega \in \Omega} (\beta_{w})})  \|\langle x,x \rangle_{\mathcal{A}}\|(\frac{(1+ \lambda) \sup_{\omega \in \Omega}(\alpha_{w})}{(1-\mu) \inf_{\omega \in \Omega} (\beta_{w})})B^{*} 
	\end{align*}
	This give that $\{R_{w}\}_{w \in \Omega}$ is a $(C$-$C)$-controlled continuous $\ast$-$g$-frame for $End_{\mathcal{A}}^{\ast}(\mathcal{H})$.
\end{proof}
\begin{theorem}
	Let $\{T_{w}\}_{w \in \Omega}$ be a $(C$-$C)$-controlled continuous $\ast$-$g$-frame for $End_{\mathcal{A}}^{\ast}(\mathcal{H})$ with bounds $\nu$ and $\delta$. Let $\{R_{w}\}_{w \in \Omega} \in End_{\mathcal{A}}^{\ast}(\mathcal{H})$ and  $\alpha, \, \beta \geq 0$. If $0\leq \alpha + \frac{\beta}{\nu\nu^{\ast}}< 1$ such that for all $x\in \mathcal{H}$, we have  
	$$\bigg\|\int_{\Omega}\langle (T_{w}-R_{w})C x,(T_{w}-R_{w})Cx\rangle_{\mathcal{A}} d\mu({\omega})\bigg\|\leq \alpha \bigg\|\int_{\Omega}\langle T_{w}C x,T_{w}  Cx\rangle_{\mathcal{A}} d\mu({\omega})\bigg\| +\beta \|\langle  x, x\rangle_{\mathcal{A}}\|.$$
	Then $\{R_{w}\}_{w \in \Omega}$ is a $(C$-$C)$-controlled continuous $\ast$-$g$-frame  with bounds $\nu\left(1-\sqrt{\alpha +\frac{\beta}{\nu\nu^{\ast}}}\right)$ and $\delta\left(1+\sqrt{\alpha +\frac{\beta}{\nu \nu^{\ast}}}\right)$.
	
\end{theorem}
\begin{proof}
	Let $\{T_{w}\}_{w \in \Omega}$ be a $(C$-$C)$-controlled continuous $\ast$-$g$-fram with bounds $\nu$ and $\delta$. Then for any $x \in \mathcal{H}$, we have
	\begin{align*}
		\|\{T_{w}Cx\}_{w\in \Omega}\|&\leq \|\{(T_{w}-R_{w})Cx\}_{w\in \Omega}\| +\|\{R_{w}Cx\}_{w\in \Omega}\|\\
		&\leq (\alpha \bigg\|\int_{\Omega}\langle T_{w}C x,T_{w}C x\rangle_{\mathcal{A}} d\mu({\omega})\bigg\| +\beta \|\langle  x, x\rangle_{\mathcal{A}}\|)^{\frac{1}{2}}\\
		& \quad +\bigg\|\int_{\Omega}\langle R_{w}C x,R_{w}C x\rangle_{\mathcal{A}} d\mu({\omega})\bigg\|^{\frac{1}{2}}\\
		&\leq  \left( \alpha \bigg\|\int_{\Omega}\langle T_{w}C x,T_{w}C  x\rangle_{\mathcal{A}} d\mu({\omega})\bigg\| +\frac{\beta}{\nu\nu^{\ast}} \bigg\|\int_{\Omega}\langle T_{w}C x,T_{w}C  x\rangle_{\mathcal{A}} d\mu({\omega})\bigg\|\right) ^{\frac{1}{2}}\\
		&\qquad\qquad\qquad\qquad +\bigg\|\int_{\Omega}\langle R_{w} Cx,R_{w}C x\rangle_{\mathcal{A}} d\mu({\omega})\bigg\|^{\frac{1}{2}}\\
		&=\sqrt{\alpha +\frac{\beta}{\nu\nu^{\ast}}} \|\{T_{w}Cx\}_{w\in \Omega}\|+\bigg\|\int_{\Omega}\langle R_{w} Cx,R_{w}C x\rangle_{\mathcal{A}} d\mu({\omega})\bigg\|^{\frac{1}{2}}.
	\end{align*}
	Therefore
	$$\left(1-\sqrt{\alpha +\frac{\beta}{\nu\nu^{\ast}}}\right)\|\{T_{w}Cx\}_{w\in \Omega}\| \leq \bigg\|\int_{\Omega}\langle R_{w}C x,R_{w}C x\rangle_{\mathcal{A}} d\mu({\omega})\bigg\|^{\frac{1}{2}} .$$
	Thus 
	\begin{align*}
		\nu\left(1-\sqrt{\alpha +\frac{\beta}{\nu\nu^{\ast}}}\right)\|\langle  x, x\rangle_{\mathcal{A}}\|\left(1-\sqrt{\alpha +\frac{\beta}{\nu\nu^{\ast}}}\right)\nu^{\ast}
		&\leq  \bigg\|\int_{\Omega}\langle R_{w}C x,R_{w}C x\rangle_{\mathcal{A}} d\mu({\omega})\bigg\|.\\
	\end{align*}
	Also, we have 
	\begin{align*}
		\|\{R_{w}Cx\}_{w\in \Omega}\|&\leq \|\{(T_{w}-R_{w})Cx\}_{w\in \Omega}\| +\|\{T_{w}Cx\}_{w\in \Omega}\|\\
		&\leq \sqrt{\alpha +\frac{\beta}{\nu\nu^{\ast}}} \|\{T_{w}Cx\}_{w\in \Omega}\|+\|\{T_{w}Cx\}_{w\in \Omega}\|\\
		&= \left(1+\sqrt{\alpha +\frac{\beta}{\nu\nu^{\ast}}}\right) \|\{T_{w}x\}_{w\in \Omega}\|\\
		&\leq \sqrt{\delta}  \left(1+\sqrt{\alpha +\frac{\beta}{\nu\nu^{\ast}}}\right) \|\langle x,x\rangle_{\mathcal{A}}\|^{\frac{1}{2}}\sqrt{\delta^{\ast}}.
	\end{align*}
	Hence   
	$$\bigg\|\int_{\Omega}\langle R_{w}C x,R_{w}C x\rangle_{\mathcal{A}} d\mu({\omega})\bigg\|\leq \delta \left(1+\sqrt{\alpha +\frac{\beta}{\nu\nu^{\ast}}}\right) \|\langle x,x\rangle_{\mathcal{A}}\|\left(1+\sqrt{\alpha +\frac{\beta}{\nu\nu^{\ast}}}\right)\delta^{\ast}.$$
	Therefore 
	\begin{align*}
		\small \nu\left(1-\sqrt{\alpha +\frac{\beta}{\nu\nu^{\ast}}}\right)\|\langle x, x\rangle_{\mathcal{A}}\| & \left(1-\sqrt{\alpha +\frac{\beta}{\nu\nu^{\ast}}}\right)\nu^{\ast}\leq \bigg\|\int_{\Omega}\langle R_{w}C x,R_{w}C x\rangle_{\mathcal{A}} d\mu({\omega})\bigg\|\\
		&\leq  \delta \left(1+\sqrt{\alpha +\frac{\beta}{\nu\nu^{\ast}}}\right) \|\langle x,x\rangle_{\mathcal{A}}\|\delta \left(1+\sqrt{\alpha +\frac{\beta}{\nu\nu^{\ast}}}\right) \delta^{\ast}
	\end{align*}
	Hence $\{R_{w}\}_{w \in \Omega}$ is a $(C$-$C)$-controlled continuous $\ast$-$g$-frame  with bounds $\nu\left(1-\sqrt{\alpha +\frac{\beta}{\nu\nu^{\ast}}}\right)$ and $\delta\left(1+\sqrt{\alpha +\frac{\beta}{\nu\nu^{\ast}}}\right)$.	
\end{proof}
\begin{corollary}
	Let $\{T_{w}\}_{w \in \Omega}$  is a $(C$-$C)$-controlled continuous $\ast$-$g$-frame for $End_{\mathcal{A}}^{\ast}(\mathcal{H})$  with bounds $\nu$ and $\delta$. Let $\{R_{w}\}_{w \in \Omega} \subset End_{\mathcal{A}}^{\ast}(\mathcal{H}) $ and $0\leq \alpha $.
	If $0\leq \alpha <\nu$ such that 
	$$\bigg\|\int_{\Omega}\langle (T_{w}- R_{w}) x,(T_{w}-R_{w}) x \rangle_{\mathcal{A}} d\mu({\omega})\bigg\|\leq \alpha \|\langle x, x\rangle_{\mathcal{A}}\|,\,\ x \in \mathcal{H},$$
	then $\{R_{w}\}_{w \in \Omega}$ is a $(C$-$C)$-controlled continuous $\ast$-$g$-frame with bounds $\nu(1-\sqrt{{\frac{\alpha}{\nu\nu^{\ast}}}})^2$ and $\delta(1+\sqrt{{\frac{\alpha}{\nu\nu^{\ast}}}})^2$.
	
\end{corollary}
\begin{proof}
	The proof comes from the previous theorem.
	
\end{proof}

	\section{Some Properties of $(C,C^{'})$-Controlled Continuous $\ast$-$g$-frames}\label{section6}

	\begin{proposition}
		Let $\{\Lambda_{w}, w\in\Omega\}$ be a continuous $\ast$-$g$-frame for $U$ with respect to $\{V_{w}: w\in\Omega\}$ and $S$ be the continuous $\ast$-$g$-frame operator associated. Let $C,C^{'} \in GL^{+}(U)$, then $\{\Lambda_{w}, w\in\Omega\}$ is  $(C$-$C^{'})$-controlled continuous $\ast$-$g$-frames.	
	\end{proposition}
	\begin{proof}
		Let $\{\Lambda_{w}, w\in\Omega\}$ is a continuous $\ast$-$g$-frame with bounds $A$ and $B$.\\By Theorem \ref{t1}, we have
		\begin{equation}\label{eq1p4}
			\|A^{-1}\|^{-2}\|\langle x,x\rangle \| \leq\bigg\| \int_{\Omega}\langle\Lambda_{w}x,\Lambda_{w}x\rangle d\mu(w)\bigg\| \leq \|B^{2}\|\|\langle x,x\rangle \|
		\end{equation}
		again we have 
		\begin{equation} \label{eq2p4}  
			\bigg\| \int_{\Omega}\langle\Lambda_{w}Cx,\Lambda_{w}C^{'}x\rangle d\mu(w)\bigg\|=\|\langle S_{CC^{'}}x,x \rangle \|,
		\end{equation} 
		from \eqref{eq1p4} and \eqref{eq2p4}, we have 
		\begin{equation}\label{eq7}
			\bigg\|\int_{\Omega}\langle\Lambda_{w}Cx,\Lambda_{w}C^{'}x\rangle d\mu(w)\bigg\|=\|C\| \|C^{'}\|\bigg\| \int_{\Omega}\langle\Lambda_{w}x,\Lambda_{w}x\rangle d\mu(w)\bigg\|=\|C\| \|C^{'}\|\|\langle Sx,x\rangle \|
		\end{equation}
		from who precedes, we have 
		\begin{equation*}
			\|A^{-1}\|^{-2}\|\|C\| \|C^{'}\|\|x\|^{2}\leq \|\langle S_{CC^{'}}x,x\rangle \|\leq \|B\|^{2}\|C\| \|C^{'}\|\|x\|^{2}, \qquad \forall x\in U.
		\end{equation*}
		So $\{\Lambda_{w}, w\in\Omega\}$ is  $(C$-$C^{'})$-controlled continuous $\ast$-$g$-frames with bounds $\|A^{-1}\|^{-1}\|\|C\| \|C^{'}\|$ and $\|B\|\|C\| \|C^{'}\|$.
	\end{proof}
	
	\begin{theorem}
		Let  $\Lambda= \{\Lambda_{w}\in End_{\mathcal{A}}^{\ast}(U,V_{w}): w\in\Omega\}$ be a $(C$-$C^{'})$-controlled continue $\ast$-$g$-frames for $U$ with respect to $\{V_{\omega}\}_{\omega \in \Omega}$ with bounds $A$, $B$. Let $T \in End_{\mathcal{A}}^{\ast}(U)$ be invertible and commute with $C$ and $C^{'}$, then $\{\Lambda_{w}T\}_{\omega \in \Omega}$ is a $(C$-$C^{'})$-controlled continue $\ast$-$g$-frames. 
	\end{theorem}
	\begin{proof}
		We have for all $x\in U$, $Tx\in U$
		\begin{align*}
			A\langle Tx,Tx\rangle A^{\ast} \leq \int_{\Omega}\langle\Lambda_{w}CTx,\Lambda_{w}C^{'}Tx\rangle d\mu(w) &\leq B\langle Tx,Tx\rangle B^{\ast}\\
			&\leq B\|T\|^{2}\langle x,x\rangle B^{\ast}\\
			&\leq (B\|T\|)\langle x,x\rangle (B\|T\|)^{\ast}.
		\end{align*}
		On other hand, $T$ is invertible then, there exist $0\leq m$ such that 
		$$m\langle x,x\rangle m^{\ast}\leq \langle Tx,Tx\rangle .$$
		So
		\begin{equation*}
			(Am)\langle x,x\rangle (Am)^{\ast}\leq A\langle Tx,Tx\rangle A^{\ast},
		\end{equation*}
		then 
		\begin{equation*}
			(Am)\langle x,x\rangle (Am)^{\ast}\leq \int_{\Omega}\langle\Lambda_{w}CTx,\Lambda_{w}C^{'}Tx\rangle d\mu(w) \leq (B\|T\|)\langle x,x\rangle (B\|T\|)^{\ast},
		\end{equation*}
		this show that $\{\Lambda_{w}T\}_{\omega \in \Omega}$ is a $(C$-$C^{'})$-controlled continue $\ast$-$g$-frames.
	\end{proof}
	
	\begin{lemma}
		Let $C,C^{'} \in GL^{+}(U)$  and $\{\Lambda_{w}\}_{w\in\Omega} ,\{\theta_{w}\}_{w\in\Omega} \subset End^{*}_{A}(U,V_{\omega})$ be a $(C^{'})^{2}$ and  $C^{2}$-controlled continuous $\ast$-$g$-Bessel sequences for $U$ respectively, let $\{\Gamma_{w}\}_{w\in\Omega} \subset l^{\infty}(\mathbb{C})$, the operator
		$L_{\Gamma , C, \theta , \Lambda , C^{'}} : U \longrightarrow U$ defined by  $L_{\Gamma , C, \theta , \Lambda , C^{'}}x=\int_{ \Omega}\Gamma_{w}C\theta^{\ast}_{w}\Lambda_{w}C^{'}xd\mu(w)$ is well defined and bounded operator.
	\end{lemma}
	\begin{proof}
		Let $\{\Lambda_{w}\}_{w\in\Omega}$ and $\{\theta_{w}\}_{w\in\Omega}$  be a $(C^{'})^{2}$ and  $C^{2}$-controlled continuous $\ast$-$g$-Bessel sequences for $U$ respectively withs bounds $B$ and $B^{'}$ respectively.\\
		For any $x,y \in U$, we have
		\begin{align*}
			\bigg\|\int_{ \Omega}\Gamma_{w}C\theta^{\ast}_{w}\Lambda_{w}C^{'}xd\mu(w)\bigg\|^{2}&=\underset{y \in U, \|y\|=1}{\sup}\bigg\| \langle \int_{ \Omega}\Gamma_{w}C\theta^{\ast}_{w}\Lambda_{w}C^{'}xd\mu(w),y\rangle \bigg\|^{2}\\
			&=\underset{y \in U, \|y\|=1}{\sup}\bigg\|  \int_{ \Omega} \langle\Gamma_{w}\Lambda_{w}C^{'}xd\mu(w),\theta_{w}Cy\rangle \bigg\|^{2}\\
			&\leq \underset{y \in U, \|y\|=1}{\sup}\bigg\|  \int_{ \Omega}\langle \Gamma_{w}\Lambda_{w}C^{'}x,\Gamma_{w}\Lambda_{w}C^{'}x\rangle d\mu(w) \bigg\|\bigg\|  \int_{ \Omega}\langle \theta_{w}Cy,\theta_{w}Cy\rangle d\mu(w)\bigg\|^{2}.
		\end{align*}
		Since
		\begin{align*}
			\int_{ \Omega}\langle \Gamma_{w}\Lambda_{w}C^{'}x,\Gamma_{w}\Lambda_{w}C^{'}x\rangle d\mu(w) &=  \int_{ \Omega}|\Gamma_{w}|^{2}\langle \Lambda_{w}C^{'}x,\Lambda_{w}C^{'}x\rangle d\mu(w)\\
			&\leq \|\Gamma_{w}\|^{2}_{\infty} \int_{ \Omega}\langle \Lambda_{w}C^{'}x,\Lambda_{w}C^{'}x\rangle d\mu(w)\\
			&\leq \|\Gamma_{w}\|^{2}_{\infty}B\langle x,x\rangle B^{\ast}.
		\end{align*}
		Hence
		\begin{align*}
			\bigg\|\int_{ \Omega}\Gamma_{w}C\theta^{\ast}_{w}\Lambda_{w}C^{'}xd\mu(w)\bigg\|^{2} &\leq \underset{y \in U, \|y\|=1}{\sup} \|\Gamma_{w}\|^{2}_{\infty}\|B\|^{2}\|x\|^{2}\|B^{'}\|^{2}\|y\|^{2}\\
			&\leq \|\Gamma_{w}\|^{2}_{\infty}\|B\|^{2}\|x\|^{2}\|B^{'}\|^{2},
		\end{align*}
		this show that $L_{\Gamma , C, \theta , \Lambda , C^{'}}$ is well defined and bounded.
	\end{proof}

\end{document}